\documentclass[11pt]{article}

\usepackage[margin=1in]{geometry}
\usepackage{amsmath}
\usepackage{amsfonts}
\usepackage{amssymb}
\usepackage{graphicx}
\usepackage{amsthm}
\usepackage{mathdots}
\usepackage[usenames,dvipsnames]{xcolor}
\usepackage{tikz}
\usepackage{color}
\usepackage{hyperref}
\hypersetup{colorlinks=true,
    linkcolor=BrickRed,
    citecolor=OliveGreen,
    urlcolor=MidnightBlue}

\theoremstyle{plain}
\newtheorem{thm}{Theorem}[section]
\newtheorem{proposition}[thm]{Proposition}

\newtheorem{theorem}[thm]{Theorem}
\newtheorem{lemma}[thm]{Lemma}

\newtheorem{corollary}[thm]{Corollary}
\newtheorem{example}[thm]{Example}
\theoremstyle{definition}
\newtheorem{definition}[thm]{Definition}

\newtheorem{remark}[thm]{Remark}

\newcommand{\RR}{\mathbb{R}}
\newcommand{\R}{\mathbb{R}}
\newcommand{\Sym}{\mathbb{S}}
\newcommand{\tr}{\textup{tr}}
\newcommand{\spf}{\mathcal{S}}
\newcommand{\F}{\mathcal{F}}
\newcommand{\Fexp}{\mathcal{F}^{\textup{exp}}}
\newcommand{\lc}{\mathcal{H}}
\newcommand{\Lin}{\mathfrak{L}}
\newcommand{\X}{\mathcal{X}}

\newcommand{\N}{\mathcal{N}}
\newcommand{\cl}{\mathcal{L}}
\newcommand{\C}{\mathcal{C}}
\newcommand{\cfd}{\mathcal{K}}
\newcommand{\psd}{\succeq}

\newcommand{\loc}[2]{\textup{Loc}_{#2}(#1)}
\newcommand{\fs}{\mathcal{O}}

\title{Terracini Convexity}
\author{James Saunderson\thanks{Department of Electrical and Computer Systems Engineering,
Monash University, VIC 3800, Australia. \texttt{james.saunderson@monash.edu}}
\and Venkat Chandrasekaran\thanks{Department of Computing and Mathematical Sciences and Department of Electrical Engineering,
California Institute of Technology, Pasadena, CA 91125, USA. \texttt{venkatc@caltech.edu}}}

\begin{document}
\maketitle

\begin{abstract}
We present a generalization of the notion of neighborliness to non-polyhedral convex cones.  Although a definition of neighborliness is available in the non-polyhedral case in the literature, it is fairly restrictive as it requires all the low-dimensional faces to be polyhedral.  Our approach is more flexible and includes, for example, the cone of positive-semidefinite matrices as a special case (this cone is not neighborly in general).  We term our generalization Terracini convexity due to its conceptual similarity with the conclusion of Terracini's lemma from algebraic geometry.  Polyhedral cones are Terracini convex if and only if they are neighborly.  More broadly, we derive many families of non-polyhedral Terracini convex cones based on neighborly cones, linear images of cones of positive-semidefinite matrices, and derivative relaxations of Terracini convex hyperbolicity cones.  As a demonstration of the utility of our framework in the non-polyhedral case, we give a characterization based on Terracini convexity of the tightness of semidefinite relaxations for certain inverse problems.

\end{abstract}

\section{Introduction}
\label{sec:intro}

The combinatorial view of polytopes is a pillar of polyhedral theory which has played a prominent role both in deepening our understanding of the structure of polytopes as well as in illuminating those attributes of polytopes that are significant in the context of particular applications such as linear programming.  A parallel perspective for non-polyhedral convex sets -- even in the presence of additional structure -- has generally been lacking.  This limitation may be attributed to the fact that the central object of study in polyhedral combinatorics is the face lattice, and consequently, many of the key ideas and definitions in the field are face-centric.  However, face-centric notions do not always carry over naturally to the non-polyhedral setting for a number of reasons; in particular, non-polyhedral closed convex sets consist of infinitely many faces, may contain non-exposed faces, may lack faces of all dimensions, may not be closed under linear images, and so forth.  Motivated by this broad challenge of bridging the gap in our understanding between the polyhedral and non-polyhedral cases, we focus in this article on the question of obtaining a suitable generalization of \emph{neighborliness} for non-polyhedral convex sets, with a less face-centric reformulation of neighborliness of polytopes playing a central role in our development.

A polyhedral cone that is pointed is called \emph{$k$-neighborly} if the cone over any subset of up to $k$ extreme rays forms a face \cite{grunbaum2003polytopes}.\footnote{Neighborliness is usually defined for convex polytopes but it is more convenient in this article to consider the polyhedral cones.  One can recover equivalent notions for compact convex sets by taking bases of convex cones that are closed and pointed.}  Neighborliness arises in many contexts in geometry and polyhedral combinatorics, most notably in the characterization of various extremal classes of polytopes \cite{grunbaum2003polytopes} and in conditions under which linear programming relaxations are tight for certain nonconvex inverse problems \cite{donohotanner2005neighborliness}.

\subsection{Motivation}

We are aware that there is a definition available for non-polyhedral $k$-neighborly convex cones that are closed and pointed which parallels the polyhedral setting \cite{kalai2008neighborly} -- that is, the cone over any subset of up to $k$ extreme rays forms an exposed face.  However, this notion is too restrictive in the non-polyhedral case as it essentially requires that all the low-dimensional faces are polyhedral, and, in particular, are linearly isomorphic to orthants.  This limitation restricts the utility of neighborliness in the non-polyhedral context in a number of ways.

As one example, the cone of positive semidefinite matrices is not $k$-neighborly for any $k > 1$ as all the faces other than the extreme rays are non-polyhedral, and as a consequence, neighborliness is not useful for characterizing tightness of semidefinite relaxations for nonconvex problems that are ubiquitous in many applications \cite{rfp2010nuclear,crpw2012atomic}, in contrast to the situation with linear programming.  Concretely, Donoho and Tanner \cite{donohotanner2005neighborliness} used neighborliness of polytopes to characterize the exactness of linear programming relaxations for identifying nonnegative vectors with the smallest number of nonzeros in affine spaces.  A similar characterization of the success of semidefinite relaxations for identifying low-rank positive semidefinite matrices in affine spaces -- a problem that arises in a range of applications such as factor analysis, collaborative filtering, and phase retrieval, and contains NP-hard problems as special cases -- has been lacking.  Thus, we seek a more flexible notion for non-polyhedral cones that specializes to the usual definition of neighborliness for polyhedral cones.

In a different vein, the utility of neighborliness lies in the fact that it provides a succinct characterization of the geometry of the `most singular' pieces of the boundary of a polyhedral cone.  It is of intrinsic interest to understand such geometry more generally for other families of structured cones.  Hyperbolicity cones serve as an instructive case study in this regard.  These are convex cones derived from hyperbolic polynomials, with the nonnegative orthant and the positive semidefinite matrices being prominent examples.  Relaxations based on derivatives of hyperbolicity cones offer the prospect of computationally less expensive approaches for obtaining bounds on conic optimization problems with respect to hyperbolicity cones, and an intriguing feature of these relaxations is that they tend to preserve the low-dimensional faces of the original hyperbolicity cone.  Formalizing and quantifying this assertion by leveraging the perspective of neighborliness would provide new insights into the facial geometry of a large class of structured convex cones.

In this paper, we describe a generalization of neighborliness for non-polyhedral cones that addresses the preceding objectives.

\subsection{Towards a Definition for Non-Polyhedral Cones}

In aiming at an appropriate generalization of neighborliness for non-polyhedral cones that overcomes the limitation of polyhedrality of the low-dimensional faces, a natural approach is to reformulate neighborliness via other geometric attributes that are less face-centric.  As a first attempt, for a convex cone $\C$ that is closed and pointed but not necessarily polyhedral, let $\spf_\C(x)$ denote the linear span of the smallest exposed face of $\C$ that contains $x$.  Then one can check that if the extreme rays of $\C$ are exposed, $k$-neighborliness of $\C$ is equivalent to the following condition for any collection $x^{(1)},\dots,x^{(k)}$ of generators of the extreme rays of $\C$:
\begin{equation}
\spf_\C\left(\sum_{i=1}^k x^{(i)}\right) = \sum_{i=1}^k \spf_\C\left(x^{(i)}\right). \label{eq:spfterr}
\end{equation}
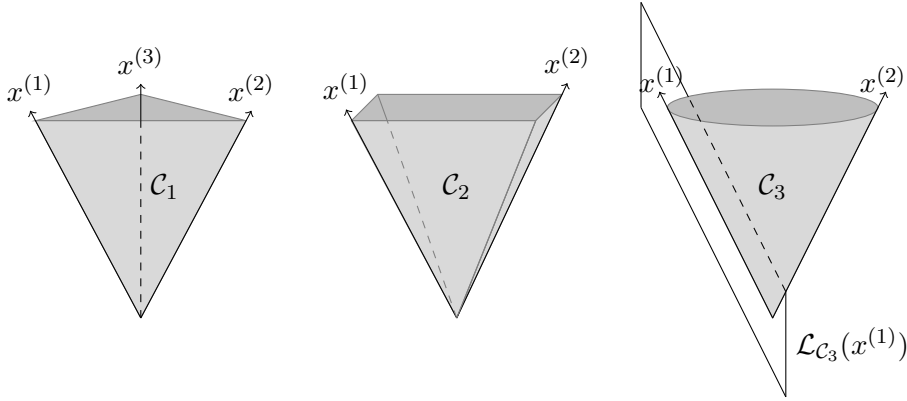
\begin{figure}
\begin{center}

  \begin{tikzpicture}[scale=0.7]

\begin{scope}[xshift=-6cm]
	\draw[gray, fill=gray!50!white] (2,-0.25) -- (-2,-0.25) -- (0,0.25) -- (2,-0.25);
	\draw[gray, fill=gray!30!white] (0,-4) -- (2,-0.25) -- (-2,-0.25) -- (0,-4);
	\draw[->] (0,-4) -- (2+0.1,-0.25+3.75*0.05);
	\draw[->] (0,-4) -- (-2-0.1,-0.25+3.75*0.05);
	\draw[dashed] (0,-4) -- (0,-0.25);
	\draw[->] (0,-0.25) -- (0,0.45);
	\node[above] (x2) at (2+0.1,-0.25+3.75*0.05) {$x^{(2)}$};
	\node[above] (x1) at (-2-0.1,-0.25+3.75*0.05) {$x^{(1)}$};
	\node[above] (x3) at (0,0.45) {$x^{(3)}$};
     
	\node[right] (C1) at (0,-1.5) {$\mathcal{C}_1$};
\end{scope}

     \draw[gray,fill=gray!50!white] (-2,-0.25) -- (-1.5,0.25) -- (2,0.25) -- (1.5,-0.25) -- (-2,-0.25);
     \draw[gray,fill=gray!30!white] (0,-4) -- (2,0.25) -- (1.5,-0.25) -- (0,-4);
     \draw[gray,fill=gray!30!white] (0,-4) -- (1.5,-0.25) -- (-2,-0.25) -- (0,-4);
     \draw[->] (0,-4) -- (-2.1,-0.25+0.05*4.25);
     \draw[gray,dashed] (0,-4) -- (-1.32,-0.25);
     \draw[gray] (-1.32,-0.25) -- (-1.5,0.25);
     \draw[->] (0,-4) -- (2.1,0.25+0.05*4.25);
     \draw[gray] (0,-4) -- (1.5,-0.25);
%
     \node[above] (x3) at (2.1,0.25+0.05*4.25) {$x^{(2)}$};
     \node[above] (x1) at (-2.1,-0.25+0.05*4.25) {$x^{(1)}$};
     \node (C2) at (0,-1.5) {$\mathcal{C}_2$};

\begin{scope}[xshift=6cm]

     \draw[gray,fill=gray!30!white] (-2,0) -- (0,-4) -- (2,0) -- (-2,0);
     \draw[gray,fill=gray!50!white] (0,0) circle(2cm and 0.35cm);


     \node (C3) at (0,-1.5) {$\mathcal{C}_3$};

     \draw[black] (-2.5,1-1) -- (-2.5,1+1);
     \draw[black] (0.25,-4.5-1) -- (0.25,-4.5+1);
     \draw[black] (0.25,-4.5-1) -- (-2.5,1-1);
     \draw[black,dashed] (0.25,-4.5+1) -- (-1.6,0.2);
     \draw[black] (-2.5,2) -- (-1.6,0.2);

\draw[->] (0,-4) -- (-2.15,0.3);
    \draw[->] (0,-4) -- (2.15,0.3);
    \node[above] (x1) at (-2.1,0.1) {$x^{(1)}$};
    \node[above] (x2) at (2.1,0.1) {$x^{(2)}$};
    \node[right] (L1) at (0.25,-4.5) {$\mathcal{L}_{\mathcal{C}_3}(x^{(1)})$};
\end{scope}

  \end{tikzpicture}

\end{center}
\caption{\label{fig:intro} Illustration of neighborliness properties of three cones.  $\C_1$ is neighborly while $\C_2$ is not.  $\C_3$ is not neighborly but it serves as an instructive example for the definition of Terracini convexity.}
\end{figure}
One can check that the left-hand-side of this equation always contains the right-hand-side, with the containment being strict in general and equality holding only for $k$-neighborly cones.  It is instructive to consider the three cones in $\R^3$ that are shown in Figure~\ref{fig:intro} from the perspective of the relation \eqref{eq:spfterr}.  The cone $\C_1$ is isomorphic to the orthant in $\R^3$, which is $3$-neighborly, and therefore the relation \eqref{eq:spfterr} holds for any subset of the generators of the three extreme rays.  The cone $\C_2$ is not $2$-neighborly as the cone over the generators $x^{(1)}, x^{(2)}$ is not a face of $\C_2$; accordingly, we note that $\spf_{\C_2}(x^{(1)} + x^{(2)}) \supsetneq \spf_{\C_2}(x^{(1)}) + \spf_{\C_2}(x^{(2)})$.  Finally, the ice-cream cone $\C_3$ is evidently not $2$-neighborly by considering the cone over the generators $x^{(1)}, x^{(2)}$; as expected, we again have the strict containment $\spf_{\C_3}(x^{(1)} + x^{(2)}) \supsetneq \spf_{\C_3}(x^{(1)}) + \spf_{\C_3}(x^{(2)})$.  The cone $\C_3$ presents an interesting case study as it is also linearly isomorphic to the cone of $2 \times 2$ symmetric positive semidefinite matrices.  As mentioned previously, developing a suitable generalization of neighborliness that encompasses the cone of positive semidefinite matrices is one of the motivations for this article, and we investigate next what precisely fails with the relation \eqref{eq:spfterr} for $\C_3$.

For a polyhedral cone $\C$ that is pointed, the map $\spf_\C(x)$ represents a kind of ``local linearization'' of $\C$ around the point $x$; concretely, the set $\spf_\C(x)$ is the largest subspace -- also called the \emph{lineality space} -- in the cone of feasible directions from $x$ into $\C$.  However, the interpretation of $\spf_\C(x)$ as a local linearization of $\C$ at $x$ no longer holds in general if $\C$ is not polyhedral.  For the cone $\C_3$ in Figure~\ref{fig:intro}, the set $\spf_{\C_3}(x^{(1)})$ does not fully represent a local linearization of $\C_3$ around $x^{(1)}$ as it fails to account for the curvature of the boundary of $\C_3$ at $x^{(1)}$.  Rather, the subspace $\cl_{\C_3}(x^{(1)})$ in Figure~\ref{fig:intro}, akin to a tangent space at $x^{(1)}$ with respect to the boundary of $\C_3$, provides a more accurate local linearization of $\C_3$ at $x^{(1)}$.  Letting $\cl_{\C_3}(x^{(2)})$ similarly denote an accurate local linearization of $\C_3$ at $x^{(2)}$, we observe that $\cl_{\C_3}(x^{(1)})+\cl_{\C_3}(x^{(2)}) = \R^3$.  As $x^{(1)}+x^{(2)}$ lies in the interior of $\C_3$, a natural local linearization of $\C_3$ at $x^{(1)} + x^{(2)}$ is the full space $\R^3$, i.e., $\cl_{\C_3}(x^{(1)}+x^{(2)}) = \R^3$.  Consequently, we have that the relation \eqref{eq:spfterr} holds for $\C_3$ with $k=2$ if we substitute $\spf_{\C_3}$ with $\cl_{\C_3}$.  Motivated by this discussion, our generalization of neighborliness to closed, convex, pointed cones is based on a criterion analogous to \eqref{eq:spfterr} with a more accurate notion of local linearization; as we discuss in the sequel, this criterion is satisfied by neighborly polyhedral cones, cones of positive semidefinite matrices, as well as many other families.

\subsection{Terracini Convex Cones}
We begin by giving a formal definition of the map $\cl_{\C}(x)$.  In the example with the cone $\C_3$ from Figure~\ref{fig:intro}, the set $\cl_{\C_3}(x)$ corresponds to a tangent space.  However, convex cones in general have both smooth and singular features in their boundary, and therefore we do not explicitly appeal to any differential notions.  Our definition is stated in terms of the feasible directions $\cfd_\C(x)$ into a convex cone $\C \subset \R^d$ that is closed and pointed from any $x \in \C$:
\begin{equation*}
\cfd_\C(x) = \text{cone}\{z - x \;:\; z \in \C\}.
\end{equation*}
The closure of the cone of feasible directions $\overline{\cfd_\C(x)}$ is called the \emph{tangent cone} of $\C$ at $x$.
\begin{definition}
Let $\C \subset \R^d$ be a convex cone that is closed and pointed.  For any $x \in \C$, the \emph{convex tangent space} of $\C$ at $x$ is denoted by $\cl_\C(x)$ and is defined as the lineality space of the tangent cone of $\C$ at $x$:
\begin{equation*}
\cl_\C(x) = \overline{\cfd_\C(x)} \cap -\overline{\cfd_\C(x)}.
\end{equation*}
\end{definition}
In some sense, the subspace $\cl_\C(x)$ represents all those directions from $x$ in which the cone $\C$ is locally ``flat''.  For smooth convex cones $\C$ that are closed and pointed, the convex tangent space $\cl_\C(x)$ at a point $x$ ($\neq 0$) on the boundary is indeed the tangent space with respect to the boundary of $\C$ at $x$.  For polyhedral cones $\C$ that are pointed, one can check that $\cl_\C(x) = \spf_\C(x)$. With this definition, we are in a position to present the main object of investigation of this article.
\begin{definition}
	\label{def:kterracini}
A convex cone $\C \subset \R^d$ that is closed and pointed is \emph{$k$-Terracini convex} if the following condition holds for any collection $x^{(1)}, \dots, x^{(k)}$ of generators of extreme rays of $\C$:
\begin{equation}\label{eq:kterracini}
\cl_\C\left(\sum_{i=1}^k x^{(i)}\right) = \sum_{i=1}^k \cl_\C\left(x^{(i)}\right).
\end{equation}
If $\C$ is $k$-Terracini convex for all $k$, then we say that $\C$ is \emph{Terracini convex}.
\end{definition}
One inclusion always holds as $\cl_\C\left(\sum_{i=1}^k x^{(i)}\right) \supseteq \sum_{i=1}^k \cl_\C\left(x^{(i)}\right)$, and the relevant portion of this definition is the other inclusion.  The reason for the terminology `Terracini convexity' is that the stipulation in this definition mirrors the consequence of Terracini's lemma in algebraic geometry \cite{terracini}, with convex tangent space playing the role in our context that a tangent space does in Terracini's lemma.\footnote{Consider a projective variety $\mathcal{V}$ over an algebraically closed field, and let $\mathcal{V}^{(k)}$ be the $k$'th secant variety of $\mathcal{V}$ obtained by taking the closure of the set of spans of every collection of $k$ points in $\mathcal{V}$.  Informally, Terracini's lemma states that for $k$ generic points $X^{(1)},\dots,X^{(k)} \in \mathcal{V}$, the tangent space at a generic point in the span of $\{X^{(1)},\dots,X^{(k)}\}$ with respect to $\mathcal{V}^{(k)}$ is equal to the sum of the tangent spaces at each $X^{(i)}$ with respect to $\mathcal{V}$.}  We give next some preliminary examples of $k$-Terracini convex cones:

\begin{example}
To begin with, it is instructive to compare $k$-Terracini convexity to $k$-neighborliness for polyhedral cones.  For a polyhedral cone $\C$ that is pointed, we observed previously that $\cl_\C(x) = \spf_\C(x)$ for $x \in \C$.  As $\C$ has exposed extreme rays and as the relation \eqref{eq:spfterr} is equivalent to $k$-neighborliness, we have that $k$-Terracini convexity and $k$-neighborliness are equivalent for pointed polyhedral cones.  We also prove this fact as a special case of a more general result (see Theorem~\ref{thm:neighborly} and Corollary~\ref{cor:polyneighborly}).
\end{example}

\begin{example}
All convex cones that are closed and pointed are trivially $1$-Terracini convex.  As a contrast, based on the generalization of \cite{kalai2008neighborly} of neighborliness to non-polyhedral cones, a convex cone that is closed and pointed is $1$-neighborly if and only if all its extreme rays are exposed.
\end{example}

\begin{example}
	Let $\C \subset \R^d$ be a smooth convex cone that is closed and pointed.  Then $\C$ is Terracini convex.  To see this, consider any collection $x^{(1)},\dots,x^{(k)}$ of generators of extreme rays of $\C$.  Due to the smoothness of $\C$, we have that $\sum_{i=1}^k \cl_\C\left(x^{(i)}\right) = \textup{span}(\C)$ for $k\geq 2$, unless all the $x^{(i)}$'s generate the same extreme ray (in which case the Terracini convexity condition is trivially satisfied).
\end{example}

\begin{example}
As our next example, we consider the cone of positive semidefinite matrices $\Sym^d_+$ in the space of $d \times d$ real symmetric matrices $\Sym^d$.  This cone consists of both smooth and singular features in its boundary.  For $X \in \Sym^d_+$, one can check that $\cl_{\Sym^d_+}(X) = \{MX + XM \;:\; M \in \Sym^n\}$, from which it follows that $\Sym^d_+$ is Terracini convex.  We give an alternative proof of this fact via a dual perspective on Terracini convexity; see Example~\ref{example:psd} after Proposition~\ref{prop:dual-def}.
\end{example}

It is instructive to consider the definition of Terracini convexity from a dual perspective, as this leads to a characterization that is more easily verified in some cases.  In preparation to state this dual criterion, we recall that the polar of a cone $\mathcal{S} \subset \R^d$ is the collection of linear functionals that are nonpositive on $\mathcal{S}$ and is denoted $\mathcal{S}^\circ$.  With this notation, the \emph{normal cone} to a convex cone $\C \subset \R^d$ at $x \in \C$ is denoted $\N_\C(x)$ and is the polar $\cfd_\C(x)^\circ$ of the cone of feasible directions from $x$ into $\C$.  As $\C$ is a cone, one can check that the normal cone to $\C$ at $x \in \C$ is given by:
\begin{equation} \label{eq:normalcone}
\N_{\C}(x) = \cfd_{\C}(x)^\circ = \{\ell \in \C^\circ ~:~ \ell(x) = 0\},
\end{equation}
which is the set of linear functionals that are nonpositive on $\C$ and vanish at $x$.  We now establish an equivalent dual formulation of Terracini convexity.
\begin{proposition} \label{prop:dual-def}
A closed, pointed, convex cone $\C \subset \R^d$ is $k$-Terracini convex if and only if for any collection $x^{(1)},\ldots,x^{(k)}$ of generators of extreme rays of $\C$,
\begin{equation} \label{eq:dual-def}
\textup{span}\left(\bigcap_{i=1}^{k}\N_{\C}(x^{(i)})\right) = \bigcap_{i=1}^{k}\textup{span}\left(\N_{\C}(x^{(i)})\right).
\end{equation}
\end{proposition}

\begin{remark}
	\label{rmk:inclusion}
In the result above, one inclusion is trivial -- we always have that the span of the
intersection of the normal cones is contained inside the intersection of the spans of the normal cones.  Terracini convexity corresponds to the reverse inclusion being true, and this is all we need to verify.  This remark is dual to the assertion after Definition~\ref{def:kterracini} about one inclusion always being true.
\end{remark}

\begin{proof}
The normal cone and the closure of the cone of feasible directions at a point $x \in \C$
are related via $\N_{\C}(x) = \cfd_{\C}(x)^\circ = \overline{\cfd_{\C}(x)}^\circ$, which implies that $\cl_{\C}(x)^\perp = \textup{span}(\N_{\C}(x))$.  Taking orthogonal complements in the definition of $k$-Terracini convexity, we see that $\C$ is $k$-Terracini convex if and only if for any collection $x^{(1)},\ldots,x^{(k)}$ of generators of extreme rays of $\C$,
\begin{equation} \label{eq:ncdef}
\textup{span}\left(\N_{\C}\left(\sum_{i=1}^kx^{(i)}\right)\right) = \bigcap_{i=1}^k \textup{span}\left(\N_{\C}(x^{(i)})\right).
\end{equation}
Here we have used that the orthogonal complement of a sum of subspaces is the intersection
of the orthogonal complements.  To complete the proof, we note that
$\N_{\C}\left(\sum_{i=1}^k x^{(i)}\right) = \bigcap_{i=1}^k \N_{\C}(x^{(i)})$
whenever $x^{(1)},\ldots,x^{(k)}\in \C$. For one inclusion, if $\ell \in \C^\circ$
and $\ell(x^{(i)}) = 0$ then $\ell\left(\sum_{i=1}^k x^{(i)}\right)=0$.
For the other inclusion, if $\ell \in \C^\circ$ and
$\ell\left(\sum_{i=1}^k x^{(i)}\right) = \allowbreak \sum_{i=1}^k \ell(x^{(i)}) = \allowbreak 0$, then we have that $\ell(x^{(i)}) \leq 0$ for each $i$ (as $\ell \in \C^\circ$) and therefore $\ell(x^{(i)}) = 0$ for each $i$ (as $\sum_{i=1}^k \ell(x^{(i)}) = 0$).
\end{proof}

To illustrate the utility of this dual formulation, we show that the positive semidefinite cone is Terracini convex.
\begin{example}[Positive semidefinite cone]\label{example:psd}
Let $\C = \Sym^d_+$ be the cone of $d\times d$ positive semidefinite matrices.  Given an extreme ray $vv'$ for $v \in \R^d$, the corresponding normal cone from \eqref{eq:normalcone} is $\N_{\C}(vv') = \{Q \in -\Sym^d_+ ~:~ v'Qv = 0\} = \{Q \in -\Sym^d_+ ~:~ Qv = 0\}$.  For any collection of generators of extreme rays $v^{(1)}{v^{(1)}}', \dots, v^{(k)}{v^{(k)}}'$ of $\C$ for $v^{(1)},\ldots,v^{(k)}\in \RR^d$, we have that:
\begin{equation*}
\begin{aligned}
\textup{span}\left(\bigcap_{i=1}^k \N_{\C}\left(v^{(i)}{v^{(i)}}'\right) \right) &= \{Q \in \Sym^d ~:~ Q v^{(i)} = 0, ~ i = 1,\dots,k\} \\ &= \bigcap_{i=1}^k \{Q \in \Sym^d ~:~ Q v^{(i)} = 0\}.
\end{aligned}
\end{equation*}

As $\mathrm{span}\left(\N_{\C}\left(v^{(i)}{v^{(i)}}'\right)\right) = \{Q \in \Sym^d ~:~ Q v^{(i)} = 0\}$ and as $k$ was arbitrary, it follows that $\Sym_+^d$ is Terracini convex.
\end{example}

\subsection{Outline of Contributions}
We initiate our study of Terracini convex cones by investigating the face structure of such cones.  Specifically, in Section~\ref{sec:basics} we provide two conditions for a closed, pointed, convex cone to be Terracini convex based on order-theoretic properties of the faces of the cone.  The first condition states that if a cone is $k$-Terracini convex for a sufficiently large $k$, which is a function of the height of the partially ordered set of faces, then the cone is Terracini convex.  The second condition gives a necessary and sufficient characterization for a cone to be Terracini convex based on the collection of all convex tangent spaces of the cone inheriting some of the lattice structure of the subspace lattice.

From the examples in the previous subsection we see that Terracini convexity is equivalent to neighborliness for polyhedral cones, but there are many families of non-polyhedral cones that are also Terracini convex.  Thus, a natural question is to clarify the distinction between Terracini convexity and neighborliness for non-polyhedral cones.  In one direction, the cone of positive semidefinite matrices serves as an example that there are Terracini convex cones that are not neighborly.  In the other direction, we prove in Section~\ref{sec:neighborly} that subject to a non-degeneracy condition that is of the form of a quadratic growth property, $k$-neighborly cones are $k$-Terracini convex.  As a consequence of this result, we obtain that the cone over the (homogeneous) moment curve, which was studied by Kalai and Wigderson in \cite{kalai2008neighborly}, is Terracini convex; see Section~\ref{sec:tcnc} for more examples.

Next we demonstrate the utility of the notion of Terracini convexity in characterizing tightness of semidefinite relaxations for the problem of finding a positive semidefinite matrix of smallest rank in an affine space. A commonly employed heuristic to solve this problem is to compute the positive semidefinite matrix of smallest trace in the given affine space, which can be obtained via a tractable semidefinite program.  In Section~\ref{sec:inverse}, we show that the success of this heuristic is closely tied to a certain cone being Terracini convex.  Our result may be viewed as a generalization of Donoho and Tanner's result on using neighborliness to characterize the exactness of linear programming relaxations for identifying nonnegative vectors with the smallest number of nonzeros in affine spaces \cite{donohotanner2005neighborliness}.  As a by-product of our result, we obtain that `most' linear images of a cone of positive semidefinite matrices are $k$-Terracini convex, where the value of $k$ depends on the dimension of the image of the linear map; see Theorem~\ref{thm:most-tc}.

In Section~\ref{sec:hyperbolic}, we investigate the Terracini convexity properties of derivative relaxations of hyperbolicity cones.  We study conditions under which derivatives of Terracini convex hyperbolicity cones continue to be $k$-Terracini convex (for suitable $k$), and in particular the relationship between the number of derivatives and $k$.  As a consequence, we obtain new examples of Terracini convex cones, and in particular ones that are basic semialgebraic; it is instructive to contrast these examples with the ones described in Section~\ref{sec:inv-new} of linear images of cones of positive semidefinite matrices, which are semialgebraic but not necessarily basic semialgebraic.

Sections~\ref{sec:neighborly}, \ref{sec:inverse}, and \ref{sec:hyperbolic} illustrate the role that Terracini convexity plays in illuminating various aspects of the facial structure of convex cones.  In each case, we obtain new examples of Terracini convex cones in the course of our discussion.  We conclude in Section~\ref{sec:discussion} with some open questions.

\section{Order-Theoretic Conditions for Terracini Convexity}
\label{sec:basics}

In this section we discuss conditions under which a closed, pointed, convex cone is Terracini convex based on the order structure underlying the faces of a convex cone.  Section~\ref{sec:tcheight} shows that a cone that is $k$-Terracini convex for sufficiently large $k$ is Terracini convex, with the threshold value of $k$ depending on the length of the longest chain of faces of the cone.  In Section~\ref{sec:tclattice} we give a lattice-theoretic condition on the collection of lineality spaces that is necessary and sufficient for a cone to be Terracini convex.

In preparation for our discussion, we recall briefly a few relevant facts about the face structure of a convex cone.  Let $\C$ be a closed, pointed, convex cone.  A subset $\F \subseteq \C$ is a \emph{face} if $x,y \in \C$ and $x+y \in \F$ implies that $x,y \in \F$.  A face $\F \subseteq \C$ is \emph{exposed} if $\F$ can be expressed as the intersection of $\C$ and a hyperplane specified by a linear functional $\ell \in \C^\circ$, i.e., $\F = \{x \in \C \;:\; \ell(x) = 0\}$.  By convention $\C$ is itself an exposed face as one can take $\ell = 0$.  The collection of (exposed) faces of $\C$ form a partially ordered set (poset) by inclusion.  For any subset $\X \subseteq \C$, let $\F_\C(\X)$ (respectively, $\Fexp_\C(\X)$) denote the inclusion-wise minimal (exposed) face of $\C$ containing $\X$.  For any element $x \in \C$, one can check that the normal cone $\N_\C(x)$ depends only on $\Fexp_\C(x)$, which in turn depends only on $\F_\C(x)$; consequently, the convex tangent space $\cl_\C(x)$ depends only on $\Fexp_\C(x)$ and in turn $\F_\C(x)$ \cite{rockafellar2015convex}.  Formally, for any $x^{(1)}, x^{(2)} \in \C$:
\begin{align} 	
	\F_\C(x^{(1)}) = \F_\C(x^{(2)}) ~ & \Leftrightarrow ~ \Fexp_\C(x^{(1)}) = \Fexp_\C(x^{(2)})\nonumber\\
	&\Leftrightarrow ~ \N_\C(x^{(1)}) = \N_\C(x^{(2)}) \Leftrightarrow ~ \cl_\C(x^{(1)}) = \cl_\C(x^{(2)}).\label{eq:face-equivalence}
\end{align}

\subsection{Terracini Convexity and the Height of the Poset of Faces}
\label{sec:tcheight}
Given a closed, pointed, convex cone $\C$, consider a collection of points $x^{(1)}, \dots, x^{(k)} \in \C$.  For large $k$, it is possible to replace the convex tangent space $\cl_\C(\sum_{i=1}^k x^{(i)})$ by $\cl_\C(\sum_{i \in I} x^{(i)})$ for a subset $I \subseteq \{1,\dots,k\}$ that is potentially much smaller than $k$, by appealing to the observation that the convex tangent space at a point depends only on the smallest face containing the point.  This allows us to conclude that if $\C$ is $k$-Terracini convex for sufficiently large $k$, then $\C$ is Terracini convex.

We describe next the relevant terminology that we use in our result.  A collection of faces $\F^{(i)}, ~ i=1,\dots,m$ of $\C$ that satisfies $\F^{(1)} \subsetneq \cdots \subsetneq \F^{(m)}$ is called a \emph{chain} of faces.  For a closed, pointed, convex cone $\C$, let $\lc(\C)$ denote the \emph{height} of the poset of faces of $\C$, which is the length of the longest chain of faces of $\C$.  As the dimension always increases strictly along chains of faces and as any maximal-length chain of faces begins with the zero-dimensional face\footnote{We do not consider the empty set to be a face of $\C$.} $\{0\}$ and ends with $\C$, we have that $\lc(\C) \leq \textup{dim}(\C)+1$.  We have next a result that allows us to replace the convex tangent space of a large sum of elements of $\C$ by that of a smaller subset based on $\lc(\C)$:

\begin{lemma}
	\label{lem:lcf}
Let $\C$ be a closed, pointed, convex cone, and consider a collection of points $x^{(1)},\ldots,x^{(k)}\in \C$.  There exists $I \subseteq \{1,\dots,k\}$ with $|I| \leq \lc(\C)-1$ such that
$\F_{\C}\left(\sum_{i=1}^k x^{(i)}\right) = \F_{\C}\left(\sum_{i\in I}x^{(i)}\right)$.
\end{lemma}
\begin{proof}
We explicitly construct a set $I$ with $|I| \leq \lc(\C)-1$.  Set $j = 0, I_0 = \emptyset, \F_\C^{(0)} = \{0\}$.  Running sequentially through $i = 1,\dots,k$, if $x^{(i)} \notin \F_\C(I_j)$, then $(a)$ increase $j$ by one, $(b)$ set $I_j = I_{j-1} \cup \{i\}$, and $(c)$ set $\F_\C^{(j)} = \F_\C\left(\sum_{m \in I_j} x^{(m)}\right)$.

The sequence of faces $\F_\C^{(0)}, \dots, \F_\C^{(j)}$ has the property that
$\F_{\C}^{(0)} \subsetneq \cdots \subsetneq \F_{\C}^{(j)} = \allowbreak \F_{\C}\left(\{x^{(1)},\dots,x^{(k)}\}\right)$, and therefore forms a chain of faces of $\C$ of length at most $\lc(\C)$.  As $\F_{\C}\left(\{x^{(1)},\dots,x^{(k)}\}\right) = \F_{\C}\left(\sum_{i=1}^k x^{(i)}\right)$ and as the index set $I_j$ satisfies $|I_j| \leq \lc(\C)-1$, setting $I = I_j$ leads to the desired conclusion.
\end{proof}

We are now in a position to state and prove the main result of this section.

\begin{proposition}
	\label{prop:llct}
	Let $\C$ be a closed, pointed, convex cone that is $(\lc(\C)-1)$-Terracini convex.  Then $\C$ is Terracini convex.
\end{proposition}

\begin{proof}
Let $x^{(1)},\ldots,x^{(k)}$ be a collection of generators of extreme rays of $\C$.  By Lemma~\ref{lem:lcf}, we know that there exists $I\subseteq \{1,\dots,k\}$ with $|I| \leq \lc(\C)-1$ such that $\F_{\C}\left(\sum_{i = 1}^k x^{(i)}\right) = \F_{\C}\left(\sum_{i\in I}x^{(i)}\right)$.  From \eqref{eq:face-equivalence} we have that:
\begin{equation} \label{eq:llct1}
\cl_{\C}\left(\sum_{i=1}^k x^{(i)}\right) = \cl_{\C}\left(\sum_{i\in I}x^{(i)}\right).
\end{equation}
Since $\C$ is $(\lc(\C)-1)$-Terracini convex, it is $|I|$-Terracini convex and therefore
\begin{equation} \label{eq:llct2}
\cl_{\C}\left(\sum_{i\in I}x^{(i)}\right) = \sum_{i\in I} \cl_{\C}\left(x^{(i)}\right).
\end{equation}
Combining~\eqref{eq:llct1} and~\eqref{eq:llct2}, and noting that $\sum_{i=1}^k \cl_{\C}\left(x^{(i)}\right) \subseteq \cl_{\C}\left(\sum_{i=1}^k x^{(k)}\right)$ as well as $\sum_{i\in I} \cl_{\C}\left(x^{(i)}\right) \subseteq \sum_{i=1}^k \cl_{\C}\left(x^{(i)}\right)$, we conclude that $\C$ is $k$-Terracini convex.  Since $k$ was arbitrary, we have shown that $\C$ is Terracini convex.
\end{proof}

As a consequence of this result, we have the following corollary:
\begin{corollary}
Let $\C$ be a closed, pointed, convex cone that is $\textup{dim}(\C)$-Terracini convex.  Then $\C$ is Terracini convex.
\end{corollary}

\begin{proof}
This follows from the observation that $\lc(\C) \leq \textup{dim}(\C) + 1$.
\end{proof}

\subsection{Terracini Convexity and the Lattice of Subspaces}
\label{sec:tclattice}
Motivated by the order-theoretic structure underlying the faces of a closed, pointed, convex cone $\C \subset \R^d$, we consider the order-theoretic aspects of the collection of convex tangent spaces associated to $\C$:
\begin{equation*}
\Lin(\C) = \{\cl_\C(x) \;:\; x \in \C\}
\end{equation*}
As $\Lin(\C)$ is a subset of the collection of subspaces in $\R^d$, one may view $\Lin(\C)$ as a poset by inclusion.  However, the collection of all subspaces in $\R^d$ additionally forms a lattice (called the subspace lattice in $\R^d$) with the join of two subspaces given by their sum and the meet given by their intersection.  In this section we relate Terracini convexity of $\C$ to $\Lin(\C)$ inheriting some of the lattice structure of the collection of all subspaces in $\R^d$.

In preparation to present this result, we discuss next a link between the elements of $\Lin(\C)$ and the exposed faces of $\C$.  As noted previously in \eqref{eq:face-equivalence}, the convex tangent space at a point $x \in \C$ depends only on the smallest exposed face of $\C$ containing $x$ so that the elements of $\Lin(\C)$ are in one-to-one correspondence with the exposed faces of $\C$.  The next result describes how one obtains an exposed face of $\C$ given an element of $\Lin(\C)$:

\begin{lemma} \label{lem:face-cl-relation}
Let $\C$ be a closed, pointed, convex cone.  For any $x \in \C$ we have that:
\begin{equation*}
\Fexp_\C(x) = \C \cap \cl_\C(x).
\end{equation*}
\end{lemma}

\begin{proof}
One can check that $\Fexp_\C(x) \subseteq \cl_\C(x)$, and therefore $\Fexp_\C(x) \subseteq \C \cap \cl_\C(x)$.  In the other direction, we begin by observing that any hyperplane supporting $\C$ that contains $\Fexp_\C(x)$ must contain $\cl_\C(x)$.  Consider a hyperplane $H$ supporting $\C$ that exposes $\Fexp_\C(x)$, i.e., $\C \cap H = \Fexp_\C(x)$ (such a hyperplane must exist as $\Fexp_\C(x)$ is an exposed face).  As $\cl_\C(x) \subseteq H$, we have that $\C \cap \cl_\C(x) \subseteq \Fexp_\C(x)$.  This concludes the proof.
\end{proof}

With this result in hand, we are now in a position to state and prove the following proposition:

\begin{proposition}
Let $\C \subset \R^d$ be a closed, pointed, convex cone.  The cone $\C$ is Terracini convex if and only if $\Lin(\C)$ is a join sub-semilattice of the lattice of all subspaces in $\R^d$ (i.e., the poset $\Lin(\C)$ has a join given by the sum of two subspaces).
\end{proposition}

\begin{proof}
	Suppose first that $\C$ is Terracini convex.  Consider any pair $\cl_\C(x), \cl_\C(y) \in \Lin(\C)$ corresponding to $x,y \in \C$, and let $x = \sum_i x^{(i)}$ and $y = \sum_j y^{(j)}$ be decompositions in terms of generators of extreme rays of $\C$.  As $\C$ is Terracini convex, we have that:
\begin{align*}
	\cl_{\C}\left(x\right) + \cl_\C\left(y\right)& = \sum_i \cl_{\C}\left(x^{(i)}\right) + \sum_j \cl_\C\left(y^{(j)}\right)\\& = \cl_{\C}\left(\sum_i x^{(i)} + \sum_j y^{(j)}\right) = \cl_\C(x + y).
\end{align*}
Since $\cl_\C(x + y) \in \Lin(\C)$, the poset $\Lin(\C)$ is a join sub-semilattice of the lattice of all subspaces in $\R^d$.

In the other direction, suppose that the poset $\Lin(\C)$ is a join sub-semilattice of the lattice of all subspaces in $\R^d$.  Consider any collection $x^{(1)},\dots,x^{(k)} \in \C$ of generators of extreme rays of $\C$.  As the join is given by subspace sum, we have that $\sum_{i=1}^k \cl_\C\left(x^{(i)}\right) \in \Lin(\C)$, which implies that $\sum_{i=1}^k \cl_\C\left(x^{(i)}\right)$ is the convex tangent space at some point $y \in \C$. Then, from Lemma~\ref{lem:face-cl-relation} we see that $\C \cap \sum_{i=1}^k \cl_\C\left(x^{(i)}\right) = \Fexp_\C(y)$, and in particular, $\sum_{i=1}^k x^{(i)} \in \Fexp_\C(y)$ as each $x^{(i)} \in \cl_\C\left(x^{(i)}\right)$.  We also have that $\C \cap \cl_\C\left(\sum_{i=1}^k x^{(i)}\right) = \Fexp_\C\left(\sum_{i=1}^k x^{(i)}\right)$.  As $\sum_{i=1}^k \cl_\C\left(x^{(i)}\right) \subseteq \cl_\C\left(\sum_{i=1}^k x^{(i)}\right)$, we conclude that $\Fexp_\C(y) \subseteq \Fexp_\C\left(\sum_{i=1}^k x^{(i)}\right)$, which in turn implies that $\Fexp_\C(y) = \Fexp_\C\left(\sum_{i=1}^k x^{(i)}\right)$ because $\sum_{i=1}^k x^{(i)} \in \Fexp_\C(y)$.  Appealing to \eqref{eq:face-equivalence}, we can then conclude that $\sum_{i=1}^k \cl_\C\left(x^{(i)}\right) = \cl_\C\left(\sum_{i=1}^k x^{(i)}\right)$.
\end{proof}

Therefore, Terracini convexity of a cone $\C$ is linked to the poset $\Lin(\C)$ inheriting the join structure of the lattice of subspaces.  In general, $\Lin(\C)$ does not inherit the meet structure of the lattice of subspaces as the intersection of the convex tangent spaces corresponding to two exposed faces does not usually yield a convex tangent space corresponding to an exposed face of $\C$ (the positive semidefinite cone provides a counterexample); indeed, the preceding proposition makes no assumptions on the existence of a meet operation.

\section{Neighborliness and Terracini Convexity}
\label{sec:neighborly}
Terracini convexity is one approach to extend neighborliness from polyhedral cones to non-polyhedral convex cones.  As discussed in the introduction, there is already a previous notion of neighborliness available in the non-polyhedral case due to Kalai and Wigderson~\cite{kalai2008neighborly}.  In this section we investigate the relationship between these two concepts, and in particular we show that $k$-neighborly convex cones (formally defined in Section~\ref{sec:kncc}) are $k$-Terracini convex subject to mild non-degeneracy conditions.  Throughout this section we view $\RR^m$ as being equipped with an inner product (which varies based on context and is specified clearly in each case), and we define an associated set $\mathcal{S}^{m-1} \subset \RR^m$ of unit-norm elements induced by the inner product.  Doing so allows us to work with a distinguished set $\textup{ext}(\mathcal{K})\cap \mathcal{S}^{m-1}$ of normalized extreme rays of a closed, pointed, convex cone $\mathcal{K}\subseteq \RR^{m}$. 

\subsection{$k$-Neighborly Convex Cones}
\label{sec:kncc}
In~\cite{kalai2008neighborly} Kalai and Wigderson extend the notion of a neighborly polytope to define a $k$-neighborly embedded smooth manifold.  This concept serves as the point of departure for a definition of a $k$-neighborly convex cone that is expressed in convex-geometric terms with no reference to an underlying embedded manifold.
\begin{definition}
	Let $\mathcal{M}$ be a smooth manifold and let $\phi: \mathcal{M}\rightarrow \RR^m$ be an embedding of $\mathcal{M}$ in $\RR^m$.
	The image $\phi(\mathcal{M})$ is a \emph{$k$-neighborly embedded manifold} if for any collection 
	$x^{(1)},x^{(2)},\ldots,x^{(k)}$ of elements of $\phi(\mathcal{M})$, 
	there exists an affine function $\ell:\RR^m\rightarrow \RR$ such that
	$\ell(x^{(i)}) =0$ for $i=1,2,\ldots,k$ and $\ell(x) > 0$ for all $x\in \phi(\mathcal{M})\setminus\{x^{(1)},x^{(2)},\ldots,x^{(k)}\}$.
\end{definition}
This definition is a slight reformulation of that of Kalai and Wigderson and it is stated in a manner that is more convenient for our presentation.  The neighborliness of $\phi(\mathcal{M})$ clearly only depends on the convex hull of $\phi(\mathcal{M})$, which suggests the following notion of a $k$-neighborly convex cone.
\begin{definition}
 	A closed, pointed, convex cone $\mathcal{K}\subseteq \RR^{m}$ is \emph{$k$-neighborly} if for 
	every collection $x^{(1)},x^{(2)},\ldots,x^{(k)}$
	of normalized extreme rays of $\mathcal{K}$, 
 	there exists a linear functional $\ell:\RR^{m}\rightarrow \RR$
	such that $\ell(x^{(i)})=0$ for $i=1,2,\ldots,k$
 	and $\ell(x) > 0$ for all
	$x\in \textup{ext}(\mathcal{K})\cap \mathcal{S}^{m-1}\setminus \{x^{(1)},x^{(2)},\ldots,x^{(k)}\}$.
\end{definition}
It is straightforward to check that if an embedded smooth manifold $\phi(\mathcal{M})\subseteq \RR^m$ is $k$-neighborly, then the cone over $\phi(\mathcal{M})$, i.e., $\textup{cone}(\{1\}\times \phi(\mathcal{M}))\subseteq \RR^{m+1}$, is a $k$-neighborly convex cone.  A basic observation about $k$-neighborly convex cones is that all of their sufficiently low-dimensional faces are linearly isomorphic to a nonnegative orthant.

\begin{proposition}
\label{prop:nb-face}
Consider a closed, pointed, convex cone $\mathcal{K} \subseteq \R^m$ that is $k$-neighborly, and suppose $\mathcal{F}$ is a face of $\mathcal{K}$ of dimension $d \leq k$.  Then $\mathcal{F}$ is linearly isomorphic to $\RR_+^d$.
\end{proposition}
\begin{proof}
		As $\mathcal{K}$ is a closed, pointed, convex cone, so is $\mathcal{F}$.  Hence, $\mathcal{F}$ is the conic hull of its extreme rays.  Let $x^{(1)},\ldots,x^{(d)}$ be a choice of $d$ linearly independent normalized extreme rays of $\mathcal{F}$ (and hence of $\mathcal{K}$).  Let $\ell$ be a linear functional satisfying $\ell(x^{(i)}) = 0$ for $i=1,2,\ldots,d$ and $\ell(x) > 0$ for all other normalized extreme rays of $\mathcal{K}$, 
	whose existence is guaranteed due to the $k$-neighborliness of $\mathcal{K}$. 
	Let $\tilde{\mathcal{F}} = \{x\in \mathcal{K}\;:\; \ell(x) = 0\}$ be the face of $\mathcal{K}$ exposed by
	$\ell$. Since every extreme ray of $\mathcal{K}$ that belongs to $\tilde{\mathcal{F}}$ is also an extreme ray of $\tilde{\mathcal{F}}$, it follows 
	from the definition of $\ell$ that $x^{(1)},x^{(2)},\ldots,x^{(d)}$ are exactly the normalized 
	extreme rays of $\tilde{\mathcal{F}}$.
	As such, $\tilde{\mathcal{F}}$ is a closed, pointed, convex cone with
exactly $d$ linearly independent extreme rays, and therefore it must be linearly isomorphic to $\RR_+^d$.
	Finally, $\tilde{\mathcal{F}}$ and $\mathcal{F}$
are both faces of $\mathcal{K}$ such that their
	relative interiors have a point in common, so $\tilde{\mathcal{F}} = \mathcal{F}$~\cite[Corollary 18.1.2]{rockafellar2015convex}.
\end{proof}
Proposition~\ref{prop:nb-face} makes it clear that $k$-Terracini convex cones are not necessarily $k$-neighborly.
Indeed, we have seen that the positive semidefinite cone is Terracini convex, and yet its faces are not linearly isomorphic to nonnegative orthants in general.  We describe next an example that serves as a running illustration throughout this section.  This cone was considered by Kalai and Wigderson~\cite{kalai2008neighborly} in the language of neighborly manifolds.

\paragraph{Cone over the Veronese embedding}
The \emph{Veronese embedding} $\phi_{n,2d}:\RR^{n}\rightarrow \RR^{\binom{n+2d-1}{2d}}$ is defined by
the homogeneous moment map $\phi_{n,2d}(z) = (z^{\alpha})_{\alpha\in \mathcal{A}_{n,2d}}$ where
$\mathcal{A}_{n,2d} = \{\alpha\in \mathbb{N}^{n}\;:\; \sum_{i=1}^{n}\alpha_i = 2d\}$ and $z^\alpha:= \prod_{i=1}^{n}z_i^{\alpha_i}$.
We denote the cone over this embedding by
\[ \mathcal{C}_{n,2d} := \textup{cone}\{\phi_{n,2d}(z)\;:\; z\in \RR^n\}.\]
When discussing this example, we let $m=\binom{n+2d-1}{2d}$ and equip $\RR^m$ with the inner product\footnote{This inner product
is variously referred to as the apolar, Bombieri, Weyl-Bombieri, Fisher, or Calder\'on inner product.}
that satisfies
\[ \langle \phi(y),\phi(z)\rangle_{B} := \langle y,z\rangle^{2d}\quad\textup{for all $y,z\in \RR^n$}\]
where the inner product on the right is the Euclidean inner product on $\RR^n$. The norms associated with these inner products are denoted $\|\cdot\|_{B}$
and $\|\cdot\|$, respectively.
Any linear functional $\ell : \R^m \rightarrow \R$ restricted to the extreme rays of the cone $\mathcal{C}_{n,2d}$ can be interpreted as a
homogeneous polynomial of degree $2d$ in $n$ variables, i.e.,
\[ \ell(\phi_{n,2d}(z)) = \sum_{\alpha\in \mathcal{A}_{n,2d}} \ell_{\alpha}z^{\alpha}.\]
Under this interpretation, the dual cone $-\mathcal{C}_{n,2d}^\circ$ is the cone of (coefficients of) nonnegative homogeneous polynomials of degree $2d$ in $n$ variables.

\begin{example}[{Neighborliness of cones over Veronese embeddings~\cite{kalai2008neighborly}}]
\label{eg:veronese}
	The cone $\mathcal{C}_{n,2d}$ is a $d$-neighborly convex cone. To see this, consider a collection of up to $d$ 
	normalized extreme rays
\[ \{\phi_{n,2d}(z^{(1)}),\ldots,\phi_{n,2d}(z^{(d)})\}\subseteq
\textup{ext}(\mathcal{C}_{n,2d})\cap \mathcal{S}^{m-1}\]
and define the  linear functional
\[ \ell(\phi_{n,2d}(z)) =
\prod_{i=1}^{d}(\|z\|^2\|z^{(i)}\|^2 - \langle z,z^{(i)}\rangle^2).\]
	From the Cauchy-Schwarz inequality, we can see that this is a nonnegative polynomial in $z$. (In fact, it is a sum of squares.)
As such, $\ell$ defines a
linear functional that is nonnegative on the extreme rays of $\mathcal{C}_{n,2d}$, and hence on $\mathcal{C}_{n,2d}$ itself.
	Furthermore, the only normalized extreme rays at which $\ell$ vanishes are $\phi_{n,2d}(z^{(i)})$ for $i=1,2,\ldots,d$.
\end{example}

\subsection{Non-Degeneracy and Regularity of Convex Cones}
Our approach to showing that a $k$-neighborly cone is $k$-Terracini convex is based on the dual characterization of $k$-Terracini convexity from Proposition~\ref{prop:dual-def}.  Specifically, for any collection of normalized extreme rays $x^{(1)},\dots,x^{(k)}$ of a $k$-neighborly cone $\mathcal{K} \subseteq \R^m$, we wish to prove that $\bigcap_{i=1}^{k}\textup{span}\left(\mathcal{N}_{\mathcal{K}}(x^{(i)})\right) \subseteq \textup{span} \left(\bigcap_{i=1}^{k}\mathcal{N}_{\mathcal{K}}(x^{(i)})\right)$.  Our strategy is to identify an $\ell \in -\bigcap_{i=1}^{k}\mathcal{N}_{\mathcal{K}}(x^{(i)})$ such that 
\begin{equation}
	\label{eq:nb-suff}\ell + U \cap \left[\bigcap_{i=1}^{k}\textup{span}\left(\mathcal{N}_{\mathcal{K}}(x^{(i)})\right)\right] \subseteq -\bigcap_{i=1}^{k}\mathcal{N}_{\mathcal{K}}(x^{(i)})
\end{equation}
for an open set $U \subseteq \R^m$ containing the origin.  The linear functional that supports $\mathcal{K}$ at the points $x^{(1)}, \dots, x^{(k)}$, which is available to us from the definition of $k$-neighborliness, serves as a natural candidate for $\ell$.  The key issue with executing this strategy is that we need to control the extent to which any $\Delta \in \bigcap_{i=1}^{k} \textup{span}\left(\mathcal{N}_{\mathcal{K}}(x^{(i)})\right)$ perturbs $\ell$.  In particular, as $\Delta \in \bigcap_{i=1}^{k} \textup{span}\left(\mathcal{N}_{\mathcal{K}}(x^{(i)})\right)$ may be decomposed as $\Delta = \Delta^{(i)}_{+} - \Delta^{(i)}_{-}$ for each $i=1,\dots,k$, (with $\Delta^{(i)}_{+},\Delta^{(i)}_{-}\in -\mathcal{N}_{\mathcal{K}}(x^{(i)})$), we need to bound the amount that the `negative' parts $\Delta^{(i)}_{-}$ perturb $\ell$.  We consider two conditions to address this point.  The first one ensures that $\ell(x)$ grows sufficiently fast around $\{x^{(1)},x^{(2)},\ldots,x^{(k)}\}$.  The second one controls the growth of any linear functional in $-\mathcal{N}_{\mathcal{K}}(x)$ for any normalized extreme ray $x \in \mathcal{K}$.  Under these conditions -- with the second one applied to each $\Delta^{(i)}_{-}$ -- we show that $\ell$ dominates $\Delta^{(i)}_{-}$; consequently, we prove that for each $\Delta \in \bigcap_{i=1}^{k} \textup{span}\left(\mathcal{N}_{\mathcal{K}}(x^{(i)})\right)$ there exists $\gamma \neq 0$ such that $\ell + \gamma \Delta \in -\bigcap_{i=1}^{k}\mathcal{N}_{\mathcal{K}}(x^{(i)})$. The first condition is a requirement on $k$-neighborly cones and takes the form of a quadratic growth criterion, while the second one is a regularity property applicable to arbitrary closed, pointed, convex cones.  Both of these conditions are mild; for example, we show that the cone over the Veronese embedding satisfies them. (That being said, we are unaware of a method to prove that a $k$-neighborly cone is $k$-Terracini convex without these two conditions.)  We precisely describe the conditions next, and we prove in Section~\ref{sec:tcnc} that $k$-neighborly cones satisfying these conditions are $k$-Terracini convex.

\subsubsection{Non-Degenerate Neighborliness}
\label{sec:non-deg}

We present a non-degenerate extension of the notion $k$-neighborliness in which the linear functional exposing a subset of $k$ extreme rays satisfies an additional growth condition when restricted to nearby extreme rays.

\begin{definition}
	\label{def:nond}
	A closed, pointed, convex cone $\mathcal{K}\subseteq \RR^m$ is \emph{non-degenerate $k$-neighborly} if
	for every collection $x^{(1)},x^{(2)},\ldots,x^{(k)}$ of normalized extreme rays of $\mathcal{K}$, 
	there exist $\epsilon>0$, $\mu > 0$, and a linear functional $\ell : \R^m \rightarrow \R$,
	such that $\ell(x^{(i)})=0$ for $i=1,2,\ldots,k$, $\ell(x) > 0$
	for all $x\in (\textup{ext}(\mathcal{K})\cap \mathcal{S}^{m-1})\setminus \{x^{(1)},\ldots,x^{(k)}\}$,
	and
	\begin{equation}
	\label{eq:nond}
		\ell(x) \geq \mu\, \min_{i=1,2,\ldots,k}\|x-x^{(i)}\|^2\quad
	\textup{for all $x\in
		(\textup{ext}(\mathcal{K})\cap \mathcal{S}^{m-1}) \cap (\cup_{i=1}^{k}\mathcal{B}(x^{(i)},\epsilon))$}.
	\end{equation}
\end{definition}
The quadratic growth condition~\eqref{eq:nond} is a mild restriction, and it is
satisfied by the examples of $k$-neighborly convex cones we consider in this section.
\begin{example}[$k$-neighborly polyhedral cones are non-degenerate $k$-neighborly]
\label{eg:poly-nond}
	If $\mathcal{K} \subseteq \R^m$ is a $k$-neigborly polyhedral cone, then for any 
	collection $x^{(1)},x^{(2)},\ldots,x^{(k)}$ of normalized extreme rays
	there is a linear functional $\ell$ such that $\ell(x^{(i)}) = 0$ for $i=1,2,\ldots,k$ and
	$\ell(x) > 0$ for all other normalized extreme rays of $\mathcal{K}$. 
	As the set of normalized extreme rays is finite, one can choose $\epsilon$ smaller than half the minimum distance between normalized extreme rays and obtain that
	\[ (\textup{ext}(\mathcal{K})\cap \mathcal{S}^{m-1}) \cap (\cup_{i=1}^{k} B(x^{(i)},\epsilon)) = \{x^{(1)},x^{(2)},\ldots,x^{(k)}\},\] 
	which implies that~\eqref{eq:nond} is vacuously satisfied for any positive $\mu$.
\end{example}
\begin{example}[Cone $\mathcal{C}_{n,2d}$ over the Veronese embedding is non-degenerate $d$-neighborly]
\label{eg:veronese-nond}
For $y,z \in \R^n$ with unit Euclidean norm so that $\|\phi(y)\|_B = \|\phi(z)\|_B = 1$ (this is the norm associated with the Bombieri inner product on $\RR^m$), we have that
\begin{align*}
	 \tfrac{1}{2}\|\phi_{n,2d}(y)- \phi_{n,2d}(z)\|^2_B &= 1-\langle y,z\rangle^{2d} =
 (1-\langle y,z\rangle^2)(1+\langle y,z\rangle^2 + \cdots + \langle y,z\rangle^{2d-2})
	\\ &\leq d(1-\langle y,z\rangle^2).
\end{align*}
	Here, the inequality follows from the Cauchy-Schwarz inequality and the fact that $y$ and $z$ have unit
	Euclidean norm. 
For unit Euclidean norm $z^{(i)}, ~ i=1,2,\ldots,d$ and unit Euclidean norm $z \in \R^n$, the linear functional $\ell$ from Example~\ref{eg:veronese} satisfies
\begin{equation*}
\ell(\phi_{n,2d}(z)) = \prod_{i=1}^{d} (1-\langle z,z^{(i)} \rangle^2) \geq \prod_{i=1}^{d} \frac{\|\phi_{n,2d}(z)- \phi_{n,2d}(z^{(i)})\|^2_B}{2d}.
\end{equation*}
Choosing $\epsilon = \frac{1}{2}\min_{i\neq j}\|\phi(z^{(i)})-\phi(z^{(j)})\|_B > 0$, whenever
$\phi_{n,2d}(z)\in \bigcup_{i=1}^{k}\mathcal{B}(\phi_{n,2d}(z^{(i)}),\epsilon)$ and $\|z\|^2=1$ we have that
\[ \ell(\phi_{n,2d}(z)) \geq \tfrac{1}{2d} (\tfrac{\epsilon^2}{2d})^{d-1}\,\min_{i}\|\phi_{n,2d}(z) - \phi_{n,2d}(z^{(i)})\|^2_B.\]
	It follows that $\mathcal{C}_{n,2d}$ is non-degenerate $d$-neighborly.
\end{example}
Although the definition of being non-degenerate $k$-neighborly only requires quadratic growth
locally around the set of minimizers, compactness of the sphere
means that local quadratic growth implies global quadratic growth.
\begin{lemma}
\label{lem:nond-local-global}
If a closed, pointed, convex cone $\mathcal{K} \subseteq \R^m$ is non-degenerate $k$-neighborly then for 
	every collection $x^{(1)},x^{(2)},\ldots,x^{(k)}$ of normalized extreme rays of $\mathcal{K}$,
	there exists $\mu_0 > 0$, and a linear functional $\ell$,
	such that $\ell(x^{(i)})=0$ for $i=1,2,\ldots,k$ and
	\[ \ell(x) \geq \mu_0\,\min_{i=1,2,\ldots,k}\|x-x^{(i)}\|^2\quad
	\textup{for all $x\in
		\textup{ext}(\mathcal{K})\cap \mathcal{S}^{m-1}$}.\]
\end{lemma}
\begin{proof}
	Let $x^{(1)},x^{(2)},\ldots,x^{(k)}$ be a collection of normalized extreme rays of $\mathcal{K}$. 
Let $\epsilon$ and $\mu$ be the positive constants, and let
$\ell$ be the linear functional, that exist because $\mathcal{K}$ is non-degenerate $k$-neighborly.
	Let 
	\[ \mathcal{W} = \{x\in \textup{ext}(\mathcal{K})\cap \mathcal{S}^{m-1}\;:\; \min_{i=1,2,\ldots,k}\|x-x^{(i)}\| < \epsilon\}\]
and let $\mathcal{W}^c = \textup{ext}(\mathcal{K})\cap \mathcal{S}^{m-1} \setminus \mathcal{W}$ be its complement
in normalized extreme rays.
By compactness of $\mathcal{W}^c$ and the fact that $\ell(x) > 0$ on $\mathcal{W}^c$,
there exists some $M>0$ such that
	\[ \ell(x) \geq M \geq \tfrac{M}{4}\min_{i=1,2,\ldots,k}\|x-x^{(i)}\|^2\quad
	\textup{for all $x\in \mathcal{W}^c$}\]
where the second inequality holds because $\|x-y\|^2\leq 4$ whenever $x,y\in \mathcal{S}^{m-1}$. Since
	\[\ell(x) \geq \mu\,\min_{i=1,2,\ldots,k} \|x-x^{(i)}\|^2\quad\textup{for all $x\in \mathcal{W}$},\]
	taking $\mu_0 = \min\{\mu, M/4\}$ completes the proof.
\end{proof}

\subsubsection{Regular Cones}
\label{sec:regular}

Our notion of regularity for a closed, pointed, convex cone requires that no linear functional in the dual cone
grows too fast around its minimizer when restricted to extreme rays.
This holds whenever the restriction of a linear functional to the extreme rays is smooth.
\begin{definition}
	\label{def:reg}
A closed, pointed, convex cone $\mathcal{K}\subseteq\RR^m$ is \emph{regular} if for each
$x_0\in \textup{ext}(\mathcal{K})$ and each $\ell\in -\mathcal{N}_{\mathcal{K}}(x_0)$,
there exist $\delta>0$ and $\nu>0$ such that
\begin{equation}
\label{eq:reg}
\ell(x) \leq \nu\|x-x_0\|^2\quad
\textup{for all $x\in (\textup{ext}(\mathcal{K}) \cap \mathcal{S}^{m-1}) \cap \mathcal{B}(x_0,\delta)$}.
\end{equation}
\end{definition}
\begin{example}[Polyhedral cones are regular]
\label{eg:poly-reg}
If $\mathcal{K} \subseteq \R^m$ is a proper polyhedral cone, then the set of normalized extreme rays is finite. Therefore, for sufficiently
small $\delta$, $(\textup{ext}(\mathcal{K}) \cap \mathcal{S}^{m-1}) \cap \mathcal{B}(x_0,\delta) = \{x_0\}$. If
$\ell\in -\mathcal{N}_{\mathcal{K}}(x_0)$, then $\ell(x_0) = 0$ and so~\eqref{eq:reg}
is vacuously satisfied for any $\nu>0$.
\end{example}
\begin{example}[Cone over the Veronese embedding is regular]
\label{eg:veronese-reg}
Suppose that $z_0\in \mathcal{S}^{n-1}$ and $\ell(\phi_{n,2d}(z))$ is nonnegative and vanishes at $z_0$. 
	Consider the nonnegative homogeneous quadratic $\|z\|^2 \|z_0\|^2 - \langle z, z_0\rangle^2$, which vanishes only on the line spanned by $z_0$. Since both $\ell(\phi_{n,2d}(z))$ and its gradient vanish at $z=z_0$, there exists $M > 0$ such that $\ell(\phi_{n,2d}(z)) \leq M (\|z\|^2 \|z_0\|^2 - \langle z, z_0\rangle^2)$ for all $z\in \mathcal{S}^{n-1}$.
Then if $z\in \mathcal{S}^{n-1}$,
\begin{align*}
 \ell(\phi_{n,2d}(z)) \leq  M(1-\langle z,z_0\rangle^2) \leq M(1-\langle z,z_0\rangle^{2d})
= \tfrac{M}{2}\|\phi_{n,2d}(z) - \phi_{n,2d}(z_0)\|^2_{B}.
\end{align*}
Since $z_0$ was arbitrary, it follows that $\mathcal{C}_{n,2d}$ is regular.
\end{example}
Although the definition of a cone being regular only bounds the growth of a linear functional on normalized extreme rays
locally around its minimizer, such a local bound can be extended to a global bound.
\begin{lemma}
\label{lem:reg-local-global}
If a closed, pointed, convex cone $\mathcal{K} \subseteq \R^m$ is regular then for each $x_0\in \textup{ext}(\mathcal{K})$ and each
$\ell \in -\mathcal{N}_{\mathcal{K}}(x_0)$ there exists $\nu_0>0$ such that
$\ell(x) \leq \nu_0\|x-x_0\|^2$ for all $x\in \textup{ext}(\mathcal{K})\cap \mathcal{S}^{m-1}$.
\end{lemma}
\begin{proof}
If $x_0\in \textup{ext}(\mathcal{K})$, the cone $\mathcal{K}$ is regular, and $\ell \in -\mathcal{N}_{\mathcal{K}}(x_0)$, then there exist $\delta>0$ and $\nu \geq 0$ such that $x\in \textup{ext}(\mathcal{K}) \cap \mathcal{S}^{m-1}$ and $\|x-x_0\|< \delta$ implies $\ell(x) \leq \nu\|x-x_0\|^2$.
If, on the other hand, $\frac{\|x_0-x\|}{\delta}\geq 1$  and $L = \max_{x\in \mathcal{S}^{m-1}}\ell(x)$ then
\[ \ell(x) = \ell(x-x_0) \leq L\delta\left(\frac{\|x-x_0\|}{\delta}\right) \leq L\delta\left(\frac{\|x-x_0\|}{\delta}\right)^2 = \tfrac{L}{\delta}\|x-x_0\|^2.\]
Choosing $\nu_0 = \max\{\nu,L/\delta\}$ completes the proof.
\end{proof}

\subsection{Terracini Convexity of Neighborly Cones}
\label{sec:tcnc}
We are now in a position to state and prove the main result of this section.
\begin{theorem}
\label{thm:neighborly}
If a closed, pointed, convex cone 
	is non-degenerate $k$-neighborly and regular, then 
	it is $k$-Terracini.
\end{theorem}
\begin{proof}
	Let $x^{(1)},x^{(2)},\ldots,x^{(k)}$
	be a collection of normalized extreme rays of a closed, pointed, 
	non-degenerate $k$-neighborly convex cone $\mathcal{K}$. 
	To establish that $\mathcal{K}$ is $k$-Terracini convex, by
	Remark~\ref{rmk:inclusion} it 
	suffices to show that $\bigcap_{i=1}^{k}\textup{span}\left(\mathcal{N}_{\mathcal{K}}(x^{(i)})\right)\subseteq\textup{span}\left(\bigcap_{i=1}^{k}\mathcal{N}_{\mathcal{K}}(x^{(i)})\right)$. As such, let $\Delta\in \bigcap_{i=1}^{k}\textup{span}\left(\mathcal{N}_{\mathcal{K}}(x^{(i)})\right)$ be arbitrary.

	Let $\ell$ be a linear functional from the definition of non-degenerate $k$-neighborliness of $\mathcal{K}$.
	Since this functional is nonnegative on $\mathcal{K}$ and vanishes on $x^{(i)}$ for $i=1,2,\ldots,k$, 
	it follows that $\ell\in -\bigcap_{i=1}^{k}\mathcal{N}_{\mathcal{K}}(x^{(i)})$. Further, from Lemma~\ref{lem:nond-local-global} there exists $\mu_0 > 0$ such that
\begin{equation}
	\label{eq:growth} \ell(x) \geq \mu_0\,\min_{i=1,2,\ldots,k}\|x-x^{(i)}\|^2\quad
 \textup{for all $x\in \textup{ext}(\mathcal{K}) \cap \mathcal{S}^{m-1}$}.
\end{equation}
	Since $\Delta\in 
	\bigcap_{i=1}^{k}\textup{span}\left(\mathcal{N}_{\mathcal{K}}(x^{(i)})\right)$, 
	for each $i$ we have a decomposition
of $\Delta$ as $\Delta = \Delta^{(i)}_+ - \Delta^{(i)}_-$ where $\Delta^{(i)}_+,\Delta^{(i)}_-\in -\mathcal{N}_{\mathcal{K}}(x^{(i)})$.
As $\mathcal{K}$ is regular, for each $i=1,2,\ldots,k$, there exists $\nu^{(i)}_0> 0$ such that
$\Delta^{(i)}_- \leq \nu^{(i)}_0\|x-x^{(i)}\|^2$ for all $x\in \textup{ext}(\mathcal{K}) \cap \mathcal{S}^{m-1}$ from Lemma~\ref{lem:reg-local-global}.
Setting $\nu_0 = \max_i\{\nu^{(i)}_0\}$ we have that 
\begin{equation}
	\label{eq:neg-growth} \Delta(x) \geq -\nu_0\,\min_{i=1,2,\ldots k}\|x-x^{(i)}\|^2\quad
\textup{for all $x\in \textup{ext}(\mathcal{K}) \cap \mathcal{S}^{m-1}$.}
\end{equation}
	If we choose $0<\gamma < \mu_0/\nu_0$ it follows from~\eqref{eq:growth} and~\eqref{eq:neg-growth} that
\begin{equation}
	(\ell + \gamma \Delta)(x) \geq (\mu_0 - \gamma\nu_0)\,\min_{i=1,2,\ldots,k}\|x-x^{(i)}\|^2\quad 
\textup{for all $x\in \textup{ext}(\mathcal{K}) \cap \mathcal{S}^{m-1}$}.
\end{equation}
	Using the fact that $\Delta(x^{(i)}) = \ell(x^{(i)}) = 0$ for $i=1,2,\ldots,k$, we can conclude that $\ell+\gamma \Delta \in -\bigcap_{i=1}^{k}\mathcal{N}_{\mathcal{K}}(x^{(i)})$.
	Since $\gamma \neq 0$, it follows that $\Delta\in \textup{span}\left(\bigcap_{i=1}^{k}\mathcal{N}_{\mathcal{K}}(x^{(i)})\right)$, and so $\mathcal{K}$ is $k$-Terracini convex.
\end{proof}

	This theorem yields two immediate corollaries based on the examples in Sections~\ref{sec:non-deg} and~\ref{sec:regular}.
\begin{corollary}\label{cor:polyneighborly}
A pointed $k$-neighborly polyhedral cone is $k$-Terracini convex.
\end{corollary}
\begin{proof}
This follows immediately from Theorem~\ref{thm:neighborly} and Examples~\ref{eg:poly-nond} and~\ref{eg:poly-reg}.
\end{proof}

\begin{corollary}
	\label{cor:tveronese}
The cone $\mathcal{C}_{n,2d}$ over the Veronese embedding is $d$-Terracini convex.
\end{corollary}
\begin{proof}
This follows immediately from Theorem~\ref{thm:neighborly} and Examples~\ref{eg:veronese-nond} and~\ref{eg:veronese-reg}.
\end{proof}

	While Corollary~\ref{cor:tveronese} holds for general cones over Veronese embeddings, for the special case of the cone over the moment curve, i.e., the case $n=2$, a stronger conclusion is possible.
\begin{corollary}\label{cor:momentcurveterracini}
The cone $\mathcal{C}_{2,2d}$ over the homogeneous moment curve is Terracini convex, i.e., is $k$-Terracini convex for all $k$.
\end{corollary}
\begin{proof}
	Let $x^{(1)},\ldots,x^{(k)}$ generate distinct extreme rays of $\mathcal{C}_{2,2d}$. Then
	there exist points $z^{(1)},\ldots,z^{(k)}\in \RR^2$ such that $\phi_{2,2d}(z^{(i)}) = x^{(i)}$ for $i=1,2,\ldots,k$ and $z^{(i)}_1z^{(j)}_2 - z^{(j)}_1z^{(i)}_2 \neq 0$ whenever $i\neq j$. In other words,
	the $z^{(i)}$ represent distinct elements of the real projective line.
	Since $\mathcal{C}_{2,2d}$ is $d$-Terracini convex, to conclude that $\mathcal{C}_{2,2d}$ is
	Terracini convex
	it suffices to show that if $k\geq d+1$ then
	\[ \bigcap_{i=1}^{k}\textup{span}\left(\mathcal{N}_{\mathcal{C}_{2,2d}}(x^{(i)})\right)
	= \{0\}\subseteq \textup{span}\left(\bigcap_{i=1}^{k}\mathcal{N}_{\mathcal{C}_{2,2d}}(x^{(i)})\right).\]

	Elements $\ell\in \mathcal{N}_{\mathcal{C}_{2,2d}}(x^{(i)})$ are exactly the
	linear functionals with the property
	that $\ell(\phi_{2,2d}(z))$ is a bivariate homogeneous polynomial of degree $2d$ that is
	non-positive and vanishes at $z^{(i)}$. As such
	$\ell\in \bigcap_{i=1}^{d}\mathcal{N}_{\mathcal{C}_{2,2d}}(x^{(i)})$ if and only if
	$\ell(\phi_{2,2d}(z))$ is a non-negative multiple of
	$p(z) = -\prod_{i=1}^{d}(z_1z^{(i)}_2 - z_2z^{(i)}_1)^2$.
	From $d$-Terracini convexity of $\mathcal{C}_{2,2d}$, it follows that
	$\ell\in \bigcap_{i=1}^{d}\textup{span}\left(\mathcal{N}_{\mathcal{C}_{2,2d}}(x^{(i)})\right)$
	if and only if $\ell(\phi_{2,2d}(z))$ is a scalar multiple of $p(z)$.

	Consider any  $\tilde{\ell}\in \bigcap_{i=1}^{k}\textup{span}\left(\mathcal{N}_{\mathcal{C}_{2,2d}}(x^{(i)})\right)$ for $k\geq d+1$ and let $q(z) = \tilde{\ell}(\phi_{2,2d}(z))$. Then $q(z) = \alpha p(z)$ for
	some scalar $\alpha$ since
	\[\tilde{\ell}\in \bigcap_{i=1}^{d}\textup{span}\left(\mathcal{N}_{\mathcal{C}_{2,2d}}(x^{(i)})\right)\subseteq\bigcap_{i=1}^{k}\textup{span}\left(\mathcal{N}_{\mathcal{C}_{2,2d}}(x^{(i)})\right).\] 
	Furthermore, $q(z^{(d+1)}) = 0$
	since $\tilde{\ell}\in \textup{span}(\mathcal{N}_{\mathcal{C}_{2,2d}}(x^{(d+1)}))$. Since
	$z^{(i)}_1z^{(j)}_2 - z^{(j)}_1z^{(i)}_2 \neq 0$ whenever $i\neq j$, this is only possible if $\alpha=0$
	and hence $\tilde{\ell} = 0$. 
\end{proof}
A natural question at this stage is whether cones $\mathcal{C}_{n,2d}$ over Veronese embeddings for $n > 2$ are also Terracini convex, rather than merely being $d$-Terracini convex.  For the case of $n=3$, this question is open, and (to the best of our knowledge) cannot be resolved given the current understanding of the structure of $\mathcal{C}_{3,2d}$.  For the case of $n=4$, the following example shows that $\mathcal{C}_{4,4}$ is not Terracini convex based on Blekherman's study of dimensional differences between faces of nonnegative polynomials and
sums of squares~\cite{blekherman2009dimensional}.

\begin{example}[{\cite[Section 2.2]{blekherman2009dimensional}}]
Consider the cone $\mathcal{C}_{4,4}$, which can be viewed as dual to nonnegative quartic forms in four variables.
Let $S = \{(1,1,0,0), \allowbreak (1,0,1,0), \allowbreak (1,0,0,1), \allowbreak (0,1,1,0), \allowbreak (0,1,0,1), \allowbreak (0,0,1,1), \allowbreak (1,1,1,1)\}$.
Blekherman shows that the face of nonnegative quartic forms in four variables that vanish on $S$ has dimension $6$, i.e.,
\begin{equation}
	\label{eq:qqtd1}
	\textup{dim}\,\textup{span}\left(\bigcap_{z\in S}\mathcal{N}_{\mathcal{C}_{4,4}}(\phi_{4,4}(z))\right) = 6.
\end{equation}
Furthermore, each of the subspaces $\textup{span}(\mathcal{N}_{\mathcal{C}_{4,4}}(\phi_{4,4}(z)))$ for $z\in S$
has codimension $4$ in the $35$-dimensional space of quartic forms in four variables. The intersection (over $z \in S$)
of these subspaces are exactly the forms that double vanish on $S$.
Consequently
\begin{equation}
	\label{eq:qqtd2}
	 \textup{dim}\left(\bigcap_{z \in S}\textup{span}\left(\mathcal{N}_{\mathcal{C}_{4,4}}(\phi_{4,4}(z))\right)\right) \geq 35 - 4|S| = 7.
\end{equation}
	In Blekherman's language, the set $S$ is not $2$-independent.
It follows from~\eqref{eq:qqtd1} and~\eqref{eq:qqtd2} that $\mathcal{C}_{4,4}$ is not $7$-Terracini convex, and hence not
Terracini convex.
\end{example}

\section{Preservation of Terracini Convexity under Linear Images}
\label{sec:inverse}

In this section, we consider the Terracini convexity properties of linear images of Terracini convex cones such as the nonnegative orthant and the positive semidefinite matrices.
We carry out our investigation by analyzing the performance of convex relaxations for nonconvex inverse problems.  Specifically, we consider the problem of finding the componentwise nonnegative vector with the smallest number of nonzero entries (i.e., nonnegative sparse vectors) in an affine space, and that of finding the smallest rank positive semidefinite matrix in an affine space.  Both of these problems arise commonly in many applications and they have been widely studied in the literature.  In Section~\ref{sec:inv-orthant} we consider sparse vector recovery and we reprove a result of Donoho and Tanner that a natural linear programming relaxation succeeds in recovering nonnegative sparse vectors in an affine space if and only if a particular linear image of the nonnegative orthant is $k$-Terracini convex for an appropriate $k$ \cite{donohotanner2005neighborliness}.  Donoho and Tanner's original proof was given in the language of neighborly polytopes.  We provide an alternate proof in Section~\ref{sec:inv-orthant} by appealing to the dual relation Proposition~\ref{prop:dual-def} as it is instructive in our subsequent analysis on recovering low-rank matrices in affine spaces.  In Section~\ref{sec:inv-psd} we prove that the success of a semidefinite programming relaxation in recovering positive semidefinite low-rank matrices implies $k$-Terracini convexity of a particular linear image of the cone of positive semidefinite matrices for a suitable $k$; in the reverse direction, we show that a `robust' analog of $k$-Terracini convexity implies success of the semidefinite relaxation.  The results in Section~\ref{sec:inv-psd} lead to a new family of non-polyhedral Terracini convex cones, which we describe in Section~\ref{sec:inv-new}.  Thus, this section supplies new examples of Terracini convex cones, and our results also highlight the utility of our definition of Terracini convexity in generalizing neighborly polyhedral cones, as the usual notion of neighborliness for non-polyhedral cones is not the right one for characterizing the performance of semidefinite relaxations for low-rank matrix recovery.

\subsection{Linear Images of the Nonnegative Orthant}
\label{sec:inv-orthant}
In applications ranging from feature selection in machine learning to recovering signals and images from a limited number of measurements, a frequently encountered question is that of finding vectors with the smallest number of nonzero entries in a given affine space.  Consider the following model problem:
\begin{equation}
\begin{aligned}
\min_{x \in \R^d} & ~~~ |\mathrm{support}(x)| \\ \mathrm{s.t.} & ~~~ Ax = b, ~ x \geq 0.
\end{aligned} \label{eq:L0}\tag{P0}
\end{equation}
Here $A : \R^d \rightarrow \R^n$ is a linear map, $b \in \R^n$, $x \geq 0$ denotes componentwise nonnegativity of $x$, and $|\mathrm{support}(x)|$ denotes the number of nonzero entries of $x$.  As solving \eqref{eq:L0} is NP-hard in general, the following tractable linear programming relaxation is the method of choice that is employed in most contexts:
\begin{equation}
\begin{aligned}
LP(A, b) = \arg\min_{x \in \R^d} & ~~~ \langle 1, x \rangle \\ \mathrm{s.t.} & ~~~ Ax = b, ~ x \geq 0.
	\end{aligned} \label{eq:L1}\tag{P1}
\end{equation}
In assessing the performance of the relaxation \eqref{eq:L1}, the usual mode of analysis is to suppose that there exists a nonnegative vector $x^\star \in \R^d$ with a small number of nonzeros such that $b = A x^\star$, and to then ask whether $x^\star$ is the unique optimal solution of \eqref{eq:L1}, i.e., whether $LP(A,A x^\star) = \{x^\star\}$.  The main
result of Donoho and Tanner \cite{donohotanner2005neighborliness} relates the success of \eqref{eq:L1} to neighborliness properties of images of the $d$-simplex $\Delta^d = \{x \in \R^d \;:\; x \geq 0, ~ \langle 1, x \rangle = 1\}$ under the map $A$.

In Theorem~\ref{thm:DT-conic}, to follow, we state a conic analog of the result in \cite{donohotanner2005neighborliness}, and we reprove it in two stages.  The proof we give
offers a template for our generalization in Section~\ref{sec:inv-psd} on relating the performance of semidefinite relaxations for low-rank matrix recovery to Terracini convexity of linear images of the cone of positive semidefinite matrices.  Our analysis relies on relating the following three properties; each of these is stated with respect to a positive integer $k$, which will be clear from context.

\begin{itemize}
\item A linear map $A : \R^d \rightarrow \R^n$ satisfies the \emph{exact recovery property} if, for each $x^\star \in \R^d_+$ with $|\mathrm{support}(x^\star)| \leq k$, the unique optimal solution of the linear programming relaxation \eqref{eq:L1} is $LP(A,Ax^\star) = \{x^\star\}$.

\item Consider a linear map $B : \R^d \rightarrow \R^N$.  The cone $B(\R^d_+)$ satisfies the \emph{unique preimage property} if, for each $x^\star \in \R^d_+$ with $|\mathrm{support}(x^\star)| \leq k$, the point $B x^\star$ has a unique preimage in $\R^d_+$.

\item Consider a linear map $B : \R^d \rightarrow \R^N$.  The cone $B(\R^d_+)$ satisfies the \emph{Terracini convexity property} if it is pointed, it has $d$ extreme rays, and it is $k$-Terracini convex.
\end{itemize}

Given these notions we state next the result of Donoho and Tanner in conic form:

\begin{theorem}
	\label{thm:DT-conic}
Consider a linear map $A : \R^d \rightarrow \R^n$ that is surjective and define the linear map $B : \R^d \rightarrow \R^{n+1}$ as $Bx = \begin{pmatrix}Ax \\ \langle 1, x \rangle \end{pmatrix}$.  Suppose that $\mathrm{null}(A) \cap \R^d_{++} \neq \emptyset$.  Fix a positive integer $k < d$.  The map $A$ satisfies the exact recovery property if and only if the cone $B(\R^d_+)$ satisfies the Terracini convexity property.
\end{theorem}

\begin{remark}
	This result is a conic analog of those in \cite{donohotanner2005neighborliness}.  The assumption
	that $A$ is surjective is to ensure a cleaner argument; if this condition is not satisfied, the proof can be adapted by restricting to the image of $A$.  Finally, the results in \cite{donohotanner2005neighborliness} do not require the condition $\mathrm{null}(A) \cap \R^d_{++} \neq \emptyset$, and they are described in terms of a property termed `outward neighborliness'.  However, the particular restriction on which we focus suffices for our purposes and leads to a simpler exposition.
\end{remark}

This result leads to two types of consequences in \cite{donohotanner2005neighborliness}.  In one direction, Donoho and Tanner leveraged results on constructions of neighborly polytopes to obtain new families of linear maps $A$ for which the linear program \eqref{eq:L1} succeeds in sparse recovery.  Conversely, by building on results in the sparse recovery literature, they constructed new families of neighborly polytopes.

Our proof proceeds in two steps and is based on the following intermediate results.

\begin{lemma}
	\label{lem:er-up-orthant}
Consider a linear map $A : \R^d \rightarrow \R^n$ and define the linear map $B : \R^d \rightarrow \R^{n+1}$ as $Bx = \begin{pmatrix}Ax \\ \langle 1, x \rangle\end{pmatrix}$.  Suppose that $\mathrm{null}(A) \cap \R^d_{++} \neq \emptyset$.  Fix a positive integer $k < d$.  The map $A$ satisfies the exact recovery property if and only if the cone $B(\R^d_+)$ satisfies the unique preimage property.
\end{lemma}

\begin{proof}
For the case $x^\star = 0$, one can check that $LP(A,0) = \{0\}$ and that the unique preimage of $0 \in \R^{n+1}$ under the map $B$ in $\R^d_+$ is also $\{0\}$.  For nonzero $x^\star$, in considering the exact recovery property and the unique preimage property, we may assume without loss of generality that $\langle 1, x^\star \rangle= 1$.  The reason for this  that $LP(A, \alpha b) = \alpha LP(A, b)$ for any $\alpha > 0$; the unique preimage property is similarly unaffected by such scaling.  With this normalization, the exact recovery property is equivalent to the fact that for any $x^\star \in \R^d_+$ with $|\mathrm{support}(x^\star)| \leq k$, the point $A x^\star$ has a unique preimage in the solid simplex $\Delta^d_0 = \{x \in \R^d \;:\; \langle 1, x \rangle \leq 1, ~ x \geq 0\}$.

Consider the implication that the exact recovery property implies the unique preimage property.
	Assume that the unique preimage property does not hold.
	Then there exists $x^\star \in \R^d_+$ with $|\mathrm{support}(x^\star)| \leq k$ and $\tilde{x} \in \R^d_+$ such that $B \tilde{x} = B x^\star, ~ \tilde{x} \neq x^\star$.  Based on the description of $B$, we can conclude that $\langle 1, \tilde{x} \rangle = 1$ and therefore $\tilde{x} \in \Delta^d$.  This violates the property that $A x^\star$ has a unique preimage in $\Delta^d_0$; hence the exact recovery property does not hold.

Conversely, consider the implication that the unique preimage property implies the exact recovery property.  Assume for the sake of a contradiction that there exists $x^\star \in \R^d_+$ with $|\mathrm{support}(x^\star)| \leq k$ and $\tilde{x} \in \Delta^d_0$ such that $A \tilde{x} = A x^\star, ~ \tilde{x} \neq x^\star$.  As $\mathrm{null}(A) \cap \R^d_{++} \neq \emptyset$, there exists $x^0 \in \Delta^d$ with $|\mathrm{support}(x^0)| = d$ such that $A x^0 = 0$.  The point $x' = (1-\langle 1, \tilde{x}\rangle ) x^0 + \tilde{x}$ has the property that $B x' = B x^\star$.  Consequently, we have that $x^\star = x' = (1-\langle 1, \tilde{x}\rangle) x^0 + \tilde{x}$, which in turn implies that $x^0$ and $\tilde{x}$ belong to the smallest face of $\R^d_+$ containing $x^\star$, i.e., $\mathrm{support}(x^0) \subseteq \mathrm{support}(x^\star)$ and $\mathrm{support}(\tilde{x}) \subseteq \mathrm{support}(x^\star)$.  However, as $|\mathrm{support}(x^0)| = d$ but $|\mathrm{support}(x^\star)| \leq k < d$, we have the desired contradiction.
\end{proof}

Our next result relates the unique preimage property to the Terracini convexity property:

\begin{proposition}
	\label{prop:up-tc-orthant}
Consider a linear map $A : \R^d \rightarrow \R^n$ and define the linear map $B : \R^d \rightarrow \R^{n+1}$ as $Bx = \begin{pmatrix}Ax \\ \langle 1, x\rangle\end{pmatrix}$.  Suppose the map $B$ is surjective.  Fix a positive integer $k$.  The cone $B(\R^d_+)$ satisfies the unique preimage property if and only if it satisfies the Terracini convexity property.
\end{proposition}

\begin{proof}
First, we give a dual reformulation of the unique preimage property.  For each $x^\star \in \R^d_+$ with $|\mathrm{support}(x^\star)| \leq k$, the property that $B x^\star$ has a unique preimage in $\R^d_+$ is equivalent to the transverse intersection condition $\mathrm{null}(B) \cap \cfd_{\R^d_+}(x^\star) = \{0\}$.  The cone $\cfd_{\R^d_+}(x^\star)$ is closed and therefore one can check that this transverse intersection condition is equivalent to $\mathrm{null}(B)^\perp \cap \mathrm{ri}(\N_{\R^d_+}(x^\star)) \neq \emptyset$.  As the nonnegative orthant is a self-dual cone, the normal cone $\N_{\R^d_+}(x^\star)$ is given by a face of $\R^d_+$ of co-dimension at most $k$.  In summary, the unique preimage property states that for any face $\Omega$ of $\R^d_+$ of co-dimension at most $k$, we have that $\mathrm{null}(B)^\perp \cap \mathrm{ri}(\Omega) \neq \emptyset$.

Second, we note that the cone $B(\R^d_+)$ is pointed by construction.  As the linear map $B$ is surjective, elements of the normal cone $\N_{B(\R^d_+)}(Bx)$ for any $x \in \R^d_+$ are in one-to-one correspondence with $\mathrm{null}(B)^\perp \cap \N_{\R^d_+}(x)$.  Consequently, by appealing to Proposition~\ref{prop:dual-def}, the cone $B(\R^d_+)$ being $k$-Terracini convex is equivalent to the condition that for any face $\Omega$ of $\R^d_+$ of co-dimension at most $k$, we have that $\mathrm{span}(\mathrm{null}(B)^\perp \cap \Omega) = \mathrm{null}(B)^\perp \cap \mathrm{span}(\Omega)$.

With these two reformulations of the unique preimage property and the Terracini convexity property in hand, we proceed to establish the desired result.

Consider the implication that the unique preimage property implies the Terracini convexity property.  Based on the unique preimage property applied to elements of $\R^d_+$ with one nonzero entry, we conclude that $B(\R^d_+)$ has $d$ extreme rays.  Let $v \in \mathrm{null}(B)^\perp \cap \mathrm{ri}(\Omega)$.  Letting $U$ be an open set in $\R^d$ containing the origin, we have that $v + \epsilon [U \cap \mathrm{null}(B)^\perp \cap \mathrm{span}(\Omega)] \subset \mathrm{null}(B)^\perp \cap \mathrm{ri}(\Omega)$ for a sufficiently small $\epsilon > 0$.  Consequently, we can conclude that $\mathrm{span}(\mathrm{null}(B)^\perp \cap \Omega) = \mathrm{null}(B)^\perp \cap \mathrm{span}(\Omega)$, which is equivalent to $B(\R^d_+)$ being $k$-Terracini convex.

Next, consider the implication that the Terracini convexity property implies the unique preimage property.  We prove this by induction on $k$.  For the base case $k=1$, as the cone $B(\R^d_+)$ has $d$ extreme rays, we have that the unique preimage property holds for $k=1$.  For $k > 1$, suppose for the sake of a contradiction that $\mathrm{null}(B)^\perp \cap \mathrm{ri}(\Omega) = \emptyset$.  Thus, there exists a face $\hat{\Omega}$ of $\R^d_+$ contained strictly in $\Omega$, i.e., $\hat{\Omega} \subsetneq \Omega$ such that $\mathrm{null}(B)^\perp \cap \hat{\Omega} = \mathrm{null}(B)^\perp \cap \Omega$.  We have the following sequence of containment relations:
\begin{equation*}
\begin{aligned}
\mathrm{null}(B)^\perp \cap \mathrm{span}(\hat{\Omega}) &\subseteq \mathrm{null}(B)^\perp \cap \mathrm{span}(\Omega) \\ &= \mathrm{span}(\mathrm{null}(B)^\perp \cap \Omega) \\ &= \mathrm{span}(\mathrm{null}(B)^\perp \cap \hat{\Omega}) \\ &\subseteq \mathrm{null}(B)^\perp \cap \mathrm{span}(\hat{\Omega}).
\end{aligned}
\end{equation*}
The first relation follows from $\hat{\Omega} \subseteq \Omega$, the second one follows from the Terracini convexity property, the third one follows from $\mathrm{null}(B)^\perp \cap \hat{\Omega} = \mathrm{null}(B)^\perp \cap \Omega$, and the final one follows from the fact that the span of the intersection of two sets is contained inside the intersection of the spans of the sets.  In conclusion, all the containments are satisfied with equality and we have that $\mathrm{null}(B)^\perp \cap \mathrm{span}(\hat{\Omega}) = \mathrm{null}(B)^\perp \cap \mathrm{span}(\Omega)$, or equivalently that:
\begin{equation}
\mathrm{null}(B) + \mathrm{span}(\hat{\Omega})^\perp = \mathrm{null}(B) + \mathrm{span}(\Omega)^\perp. \label{eq:l1contradiction}
\end{equation}
As $\R^d_+$ is a polyhedral cone, we note that $\mathrm{span}(\hat{\Omega})^\perp$ and $\mathrm{span}(\Omega)^\perp$ are themselves spans of faces of $\R^d_+$.  In particular, let $\mathcal{F}, \hat{\mathcal{F}}$ be faces of $\R^d_+$ such that $\mathcal{F} \subsetneq \hat{\mathcal{F}}$, and $\mathrm{span}(\mathcal{F}) = \mathrm{span}(\Omega)^\perp, \mathrm{span}(\hat{\mathcal{F}}) = \mathrm{span}(\hat{\Omega})^\perp$.  The relationship \eqref{eq:l1contradiction} implies that there exists a generator $\hat{x}$ of an extreme ray of $\R^d_+$ in $\hat{\mathcal{F}} \backslash \mathcal{F}$ such that $\hat{x} = (x^{(+)} - x^{(-)}) + v$ for $x^{(+)}, x^{(-)} \in \mathcal{F}$ with disjoint supports and $v \in \mathrm{null}(B)$.  Hence, we have that $B (\hat{x} + x^{(-)}) = B x^{(+)}$.  As $\mathrm{dim}(\mathcal{F}) = k$, the sum of the sizes of the supports of $\hat{x} + x^{(-)}$ and of $x^{(+)}$ is at most $k+1$.  If $x^{(+)} \neq 0$ we have a contradiction due to the inductive hypothesis.  If $x^{(+)} = 0$ we have $\langle 1, (\tilde{x} + x^{(-)})\rangle = 0$, which implies that $\hat{x} + x^{(-)} = 0$ and in turn that $\hat{x} = 0$, also a contradiction.
\end{proof}

Based on these two results, we are now in a position to prove Theorem~\ref{thm:DT-conic}.

\begin{proof}[{Proof of Theorem~\ref{thm:DT-conic}}]
	As $\mathrm{null}(A) \cap \R^d_{++} \neq \emptyset$ and $k < d$ by assumption, we can apply Lemma~\ref{lem:er-up-orthant}.  Specifically, the exact recovery property for $A$ is equivalent to the unique preimage property for $B(\R^d_+)$.

Next, in preparation to apply Proposition~\ref{prop:up-tc-orthant}, we need to verify that the linear map $B$ is surjective.  The surjectivity of $B$ is equivalent to $A$ being surjective and $1 \notin \mathrm{null}(A)^\perp$.  The former condition holds by assumption and the latter condition is in turn equivalent to $\mathrm{null}(A) \nsubseteq \mathrm{span}(1)^\perp$.  The assumption $\mathrm{null}(A) \cap \R^d_{++} \neq \emptyset$ implies that $\mathrm{null}(A) \nsubseteq \mathrm{span}(1)^\perp$.  Thus, we are in a position to apply Proposition~\ref{prop:up-tc-orthant} and obtain that the unique preimage property of the cone $B(\R^d_+)$ is equivalent to $B(\R^d_+)$ satisfying the Terracini convexity property.  This concludes the proof.
\end{proof}

\subsection{Linear Images of the Positive Semidefinite Matrices}
\label{sec:inv-psd}
The development of convex relaxations for obtaining low-rank matrices in affine spaces largely paralleled and built upon the literature on sparse recovery.  Notable examples of such problems include factor analysis and collaborative filtering.  Concretely, given an affine space in $\Sym^d$ of the form $\{X \in \Sym^d \;:\; \mathcal{A}(X) = b\}$ where $\mathcal{A} : \Sym^d \rightarrow \R^n$ is a linear map and $b \in \R^n$, consider the following optimization problem for identifying a positive-semidefinite low-rank matrix in this space:
\begin{equation}
\begin{aligned}
\min_{X \in \Sym^d} & ~~~ \mathrm{rank}(X) \\ \mathrm{s.t.} & ~~~ \mathcal{A}(X) = b, ~ X \succeq 0.
\end{aligned} \label{eq:R0}\tag{R0}
\end{equation}
As with the problem \eqref{eq:L0}, the program \eqref{eq:R0} is also NP-hard to solve in general.
Consequently, the following semidefinite relaxation is widely employed in practice:
\begin{equation}
\begin{aligned}
SDP(\mathcal{A}, b) = \arg\min_{X \in \Sym^d} & ~~~ \mathrm{tr}(X) \\ \mathrm{s.t.} & ~~~ \mathcal{A}(X) = b, ~ X \succeq 0.
\end{aligned} \label{eq:R1}\tag{R1}
\end{equation}
By analogy with
the analysis of the performance of \eqref{eq:L1}, we are interested in obtaining conditions under which the unique optimal solution of \eqref{eq:R1} with $b = \mathcal{A}(X^\star)$ for a low-rank matrix $X^\star \in \Sym^d_+$ is equal to $X^\star$, i.e., whether $SDP(\mathcal{A},\mathcal{A}(X^\star)) = \{X^\star\}$.  Our objective in the remainder of this section is to relate such exact recovery to Terracini convexity of an appropriate linear image of $\Sym^d_+$.

As with the previous subsection, our analysis is organized in terms of three properties:
\begin{itemize}
\item A linear map $\mathcal{A} : \Sym^d \rightarrow \R^n$ satisfies the \emph{exact recovery property} if for any $X^\star \in \Sym^d_+$ with $\mathrm{rank}(X^\star) \leq k$, the unique optimal solution of the semidefinite programming relaxation \eqref{eq:R1} is $SDP(\mathcal{A},\mathcal{A}(X^\star)) = \{X^\star\}$.

\item Consider a linear map $\mathcal{B} : \Sym^d \rightarrow \R^N$.  The cone $\mathcal{B}(\Sym^d_+)$ satisfies the \emph{unique preimage property} if for any $X^\star \in \Sym^d_+$ with $\mathrm{rank}(X^\star) \leq k$, the point $\mathcal{B}(X^\star)$ has a unique preimage in $\Sym^d_+$.

\item Consider a linear map $\mathcal{B} : \Sym^d \rightarrow \R^N$.  The cone $\mathcal{B}(\Sym^d_+)$ satisfies the \emph{Terracini convex property} if it is closed and pointed, its extreme rays are in one-to-one correspondence with those of $\Sym^d_+$, and it is $k$-Terracini.
\end{itemize}

In what follows, let $\fs^d = \{X\in \Sym^d\;:\; \tr(X)=1,\; X \psd 0\}$ be the spectraplex. This plays the same role as the simplex $\Delta^d$ did in Section~\ref{sec:inv-orthant}.
We are now in a position to state the main new result of this section.

\begin{theorem}
	\label{thm:psd-inv}
Consider a linear map $\mathcal{A} : \Sym^d \rightarrow \R^n$ and fix a positive integer $k < d$.
	Then the following two statements hold:
\begin{enumerate}
\item Suppose that $\mathcal{A}$ is surjective and $\mathrm{null}(\mathcal{A}) \cap \Sym^d_{++} \neq \emptyset$.  Consider the linear map $\mathcal{B}: \Sym^d \rightarrow \R^{n+1}$ defined as $\mathcal{B}(X) = \begin{pmatrix}\mathcal{A}(X) \\ \mathrm{tr}(X)\end{pmatrix}$.  If the map $\mathcal{A}$ satisfies the exact recovery property, then the cone $\mathcal{B}(\Sym^d_+)$ satisfies the Terracini convexity property.

	\item  Assume that $n > {d+1 \choose 2} - {d-k+1 \choose 2}$.  Suppose there exists an open set $\mathfrak{S}$ in the space of linear maps from $\Sym^d$ to $\R^n$ with the following properties:
		\begin{itemize}
			\item $\mathcal{A}\in \mathfrak{S}$
			\item For each $\tilde{\mathcal{A}} \in \mathfrak{S}$, the map $\tilde{\mathcal{A}}$ is surjective and satisfies $\mathrm{null}(\tilde{\mathcal{A}}) \cap \Sym^d_{++} \neq \emptyset$.
	\item For each $\tilde{\mathcal{A}} \in \mathfrak{S}$ with associated $\tilde{\mathcal{B}} : \Sym^d \rightarrow \R^{n+1}$ defined as $\tilde{\mathcal{B}}(X) = \begin{pmatrix}\tilde{\mathcal{A}}(X) \\ \mathrm{tr}(X)\end{pmatrix}$, the cone $\tilde{\mathcal{B}}(\Sym^d_+)$ satisfies the Terracini convexity property.
	\end{itemize}
	Then the map $\mathcal{A}$ satisfies the exact recovery property.

\end{enumerate}
\end{theorem}

The proof in the direction from the exact recovery property to the Terracini convexity property largely follows the same sequence of steps as the proof of the analogous direction of Theorem~\ref{thm:DT-conic}, although technical care is required due to the fact that the cone of feasible directions into the cone of positive-semidefinite matrices is not closed.  In the direction from the Terracini convexity property to the exact recovery property, we require a robust analog of the Terracini convexity property.  This condition is in the same spirit as constraint qualification type assumptions that are required in the semidefinite programming literature in order to guarantee strict complementarity \cite{renegar2001book}.

Inspired by the two-stage proof in Section~\ref{sec:inv-orthant}, we begin with the following result that parallels Lemma~\ref{lem:er-up-orthant}:

\begin{lemma}
	\label{lem:psd-inv}
Consider a linear map $\mathcal{A} : \Sym^d \rightarrow \R^n$ and define the linear map $\mathcal{B}: \Sym^d \rightarrow \R^{n+1}$ as $\mathcal{B}(X) = \begin{pmatrix}\mathcal{A}(X) \\ \mathrm{tr}(X)\end{pmatrix}$.  Suppose that $\mathrm{null}(\mathcal{A}) \cap \Sym^d_{++} \neq \emptyset$.  Fix a positive integer $k < d$.  The map $\mathcal{A}$ satisfies the exact recovery property if and only if the cone $\mathcal{B}(\Sym^d_+)$ satisfies the unique preimage property.
\end{lemma}

\begin{proof}
	As with the proof of Lemma~\ref{lem:er-up-orthant}, in considering the exact recovery property and the unique recovery property, we assume without loss of generality that $\mathrm{tr}(X^\star) = 1$.  With this normalization, the exact recovery property is equivalent to the fact that for any $X^\star \in \Sym^d_+$ with $\mathrm{rank}(X^\star) \leq k$, the point $\mathcal{A}(X^\star)$ has a unique preimage in the solid spectraplex $\fs^d_0 = \{X \in \Sym^d \;:\; \mathrm{tr}(X) \leq 1, ~ X \succeq 0\}$.

Consider the implication that the exact recovery property implies the unique preimage property.  Assume 
	that the unique preimage property does not hold. Then there exists $X^\star \in \Sym^d_+$ with $\mathrm{rank}(X^\star) \leq k$ and $\tilde{X} \in \Sym^d_+$ such that $\mathcal{B}(\tilde{X}) = \mathcal{B}(X^\star), ~ \tilde{X} \neq X^\star$.  Based on the description of $\mathcal{B}$, we can conclude that $\mathrm{tr}(\tilde{X}) = 1$ and therefore $\tilde{X} \in \fs^d$.  This violates the property that $\mathcal{A}(X^\star)$ has a unique preimage in $\fs^d_0$; hence the exact recovery property does not hold.

Conversely, consider the implication that the unique preimage property implies the exact recovery property.  Assume for the sake of a contradiction that there exists $X^\star \in \Sym^d_+$ with $\mathrm{rank}(X^\star) \leq k$ and $\tilde{X} \in \fs^d_0$ such that $\mathcal{A}(\tilde{X}) = \mathcal{A}(X^\star), ~ \tilde{X} \neq X^\star$.  As $\mathrm{null}(\mathcal{A}) \cap \Sym^d_{++} \neq \emptyset$, there exists $X^0 \in \fs^d$ with $\mathrm{rank}(X^0) = d$ such that $\mathcal{A}(X^0) = 0$.  The point $X' = (1-\mathrm{tr}(\tilde{X})) X^0 + \tilde{X}$ has the property that $\mathcal{B}(X') = \mathcal{B}(X^\star)$.  Consequently, we have that $X^\star = X' = (1-\mathrm{tr}(\tilde{X})) X^0 + \tilde{X}$, which in turn implies that $X^0$ and $\tilde{X}$ belong to the smallest face of $\Sym^d_+$ containing $X^\star$.  However, as $\mathrm{rank}(X^0) = d$ but $\mathrm{rank}(X^\star) \leq k < d$, we have the desired contradiction.
\end{proof}

The next proposition represents the main new component of the proof of Theorem~\ref{thm:psd-inv}:
\begin{proposition}
	\label{prop:psd-inv}
Consider a linear map $\mathcal{A} : \Sym^d \rightarrow \R^n$ and define the linear map $\mathcal{B} : \Sym^d \rightarrow \R^{n+1}$ as $\mathcal{B}(X) = \begin{pmatrix}\mathcal{A}(X) \\ \mathrm{tr}(X)\end{pmatrix}$.  Fix a positive integer $k$.  Then we have the following two results:
\begin{enumerate}
\item Suppose the map $\mathcal{B}$ is surjective.  If the cone $\mathcal{B}(\Sym^d_+)$ satisfies the unique preimage property, then it satisfies the Terracini convexity property.

\item Assume that $n > {d+1 \choose 2} - {d-k+1 \choose 2}$.  Suppose there exists an open set $\mathfrak{S}$ in the space of linear maps from $\Sym^d$ to $\R^n$ satisfying the following conditions:
	\begin{itemize}
		\item $\mathcal{A}\in \mathfrak{S}$
		\item For each $\tilde{\mathcal{A}}\in \mathfrak{S}$, the associated linear
			map $\tilde{\mathcal{B}} : \Sym^d
			\rightarrow \R^{n+1}$ defined as
			$\tilde{\mathcal{B}}(X) =
			\begin{pmatrix}\tilde{\mathcal{A}}(X) \\
			\mathrm{tr}(X)\end{pmatrix}$ is surjective and the cone
			$\tilde{\mathcal{B}}(\Sym^d_+)$
			satisfies the Terracini convexity property.
	\end{itemize}
		Then the cone $\mathcal{B}(\Sym^d_+)$ satisfies the unique preimage property.
\end{enumerate}
\end{proposition}

\emph{Remarks}: In the direction from the Terracini convexity property to the unique preimage property, the fact that $\Sym^d_+$ is not polyhedral, unlike $\R^d_+$, complicates matters in comparison to the proof of Proposition~\ref{prop:up-tc-orthant}.  Specifically, translated to the context of the present theorem, the reasoning up to \eqref{eq:l1contradiction} in Proposition~\ref{prop:up-tc-orthant} continues to hold, but the sentence immediately after \eqref{eq:l1contradiction} is no longer true.  As stated previously, the nature of this difficulty is akin to the lack of strict complementarity in semidefinite programs (in contrast to linear programs), thus necessitating some type of constraint qualification assumption.  The `robust Terracini' form of the assumption in the second part of this result is similar in spirit to assumptions discussed in \cite{renegar2001book} to ensure strong duality in conic programs.

\begin{proof}
We begin by presenting a dual reformulation of the unique preimage property.  For each $X^\star \in \Sym^d_+$ with $\mathrm{rank}(X^\star) \leq k$, the property that $\mathcal{B}(X^\star) \in \mathcal{C}$ has a unique preimage in $\Sym^d_+$ is equivalent to the transverse intersection condition $\mathrm{null}(\mathcal{B}) \cap \cfd_{\Sym^d_+}(X^\star) = \{0\}$.  Unlike the situation with Proposition~\ref{prop:up-tc-orthant}, the cone of feasible directions $\cfd_{\Sym^d_+}(X^\star)$ is \emph{not} closed, which presents additional complications.  We prove next that we must have $\mathrm{null}(\mathcal{B}) \cap \overline{\cfd_{\Sym^d_+}(X^\star)} = \{0\}$ by reasoning that if there exists a nonzero $M \in \mathrm{null}(\mathcal{B}) \cap \overline{\cfd_{\Sym^d_+}(X^\star)}$ then there is a low-rank matrix near $X^\star$ for which the unique preimage property does not hold.

	Concretely, suppose for the sake of a contradiction that $M \in \mathrm{null}(\mathcal{B}) \cap \overline{\cfd_{\Sym^d_+}(X^\star)}$ with $M \neq 0$.  Without loss of generality, we assume that $X^\star$ has rank $r \in \{1,\dots,k\}$ with the row/column space equal to the span of the first $r$ standard basis vectors.  For such an $X^\star$, the closure of the cone of feasible directions $\overline{\cfd_{\Sym^d_+}(X^\star)}$ takes on a convenient block-diagonal form, so that $M \in \overline{\cfd_{\Sym^d_+}(X^\star)}$ may be viewed as follows:
\begin{equation*}
M = \begin{pmatrix}
P & V' \\ V & Q
\end{pmatrix},
\end{equation*}
	with $P \in \Sym^r, V \in \R^{(n-r) \times r}, Q \in \Sym^{(n-r)}_+$.  We now construct a rank-$r$ matrix for which the unique preimage property does not hold, thus violating the given assumption.  Choose any matrix $W \in \mathbb{S}^r$ such that $W$ and $W+P$ are strictly positive definite.
	We have that the matrix $\begin{pmatrix} W & -V' \\ -V & V W^{-1} V' \end{pmatrix}$ belongs to $\mathbb{S}^d_+$ and has rank equal to $r$.  Further, we also have that the matrix $\begin{pmatrix}W+P & 0 \\ 0 & Q + V W^{-1} V' \end{pmatrix}$ lies in $\mathbb{S}^d_+$.  Consequently, we have that the matrix:
\begin{equation*}
\begin{pmatrix} W+P & 0 \\ 0 & Q + V W^{-1} V' \end{pmatrix} - \begin{pmatrix} W & -V' \\ -V & V W^{-1} V' \end{pmatrix} = \begin{pmatrix} P & V' \\ V & Q \end{pmatrix}
\end{equation*}
	lies in the cone of feasible directions from $\begin{pmatrix} W & -V' \\ -V & V W^{-1} V' \end{pmatrix}$ into $\mathbb{S}^d_+$.  
		Since $M\in \mathrm{null}(\mathcal{B})$,
	\[ \mathcal{B}\begin{pmatrix} W & -V' \\ -V & V W^{-1} V' \end{pmatrix} = \mathcal{B}\begin{pmatrix} W+P & 0 \\ 0 & Q + V W^{-1} V' \end{pmatrix}\]
		and so 	the image of the rank-$r$ matrix $\begin{pmatrix} W & -V' \\ -V & V W^{-1} V' \end{pmatrix}$ under the map $\mathcal{B}$ does not have a unique preimage in $\Sym^d_+$, which gives us the desired contradiction.  In summary, we have for each $X^\star \in \Sym^d_+$ with $\mathrm{rank}(X^\star) \leq k$ that $\mathrm{null}(\mathcal{B}) \cap \overline{\cfd_{\Sym^d_+}(X^\star)} = \{0\}$, which in turn is equivalent to $\mathrm{null}(\mathcal{B})^\perp \cap \mathrm{ri}(\N_{\Sym^d_+}(X^\star)) \neq \emptyset$.  In analogy to the case of the nonnegative orthant, the positive-semidefinite cone $\Sym^d_+$ is self-dual and the normal cone $\N_{\Sym^d_+}(X^\star)$ is given by a face of $\Sym^d_+$ of dimension at least ${d-k+1 \choose 2}$ (corresponding to positive-semidefinite matrices with row/column space orthogonal to those of $X^\star$).  Thus, the unique preimage property states that for any face $\Omega$ of $\Sym^d_+$ of dimension at least ${d-k+1 \choose 2}$, we have that $\mathrm{null}(\mathcal{B})^\perp \cap \mathrm{ri}(\Omega) \neq \emptyset$.

Next, we note that for each $\tilde{\mathcal{A}} \in \mathfrak{S}$, the associated linear map $\tilde{\mathcal{B}}$ is such that the cone $\tilde{\mathcal{B}}(\Sym^d_+)$ is closed and pointed by construction.  Further, each $\tilde{\mathcal{B}}$ is surjective by assumption.  Thus, elements of the normal cone $\N_{\tilde{\mathcal{B}}(\Sym^d_+)}(\mathcal{B}(X))$ are in one-to-one correspondence with those of $\mathrm{null}(\tilde{\mathcal{B}})^\perp \cap \N_{\Sym^d_+}(X)$ for each $X \in \Sym^d_+$.  Hence, by appealing to Proposition~\ref{prop:dual-def}, the Terracini convexity property states that for any face $\Omega$ of $\Sym^d_+$ of dimension at least ${d-k+1 \choose 2}$, we have that $\mathrm{span}(\mathrm{null}(\tilde{\mathcal{B}})^\perp \cap \Omega) = \mathrm{null}(\tilde{\mathcal{B}})^\perp \cap \mathrm{span}(\Omega)$.

With these reformulations of the unique preimage property and the Terracini convexity property, we now proceed to establish the result.

\underline{Proof of Statement $1$} To prove the first result, we begin by noting that unique preimage property applied to rank-one elements of $\Sym^d_+$ implies that the cone $\mathcal{B}(\Sym^d_+)$ has extreme rays in one-to-one correspondence with those of $\Sym^d_+$. Next, let $M \in \mathrm{null}(\mathcal{B})^\perp \cap \mathrm{ri}(\Omega)$.  Letting $U$ be an open set in $\Sym^d$ containing the origin, we have that $M + \epsilon [U \cap \mathrm{null}(\mathcal{B})^\perp \cap \mathrm{span}(\Omega)] \subset \mathrm{null}(\mathcal{B})^\perp \cap \mathrm{ri}(\Omega)$ for a sufficiently small $\epsilon > 0$.  Consequently, we can conclude that $\mathrm{span}(\mathrm{null}(\mathcal{B})^\perp \cap \Omega) = \mathrm{null}(\mathcal{B})^\perp \cap \mathrm{span}(\Omega)$, which is equivalent to the Terracini convexity condition.

\underline{Proof of Statement $2$} Next we consider the second statement.  Fix a face $\Omega$ of $\Sym^d_+$ of co-dimension at least ${d-k+1 \choose 2}$.  Suppose for the sake of a contradiction that $\mathrm{span}(\mathrm{null}(\mathcal{B})^\perp \cap \mathrm{ri}(\Omega)) = \emptyset$.  As $n > {d+1 \choose 2} - {d-k+1 \choose 2}$, we have that $\mathrm{null}(\mathcal{B})^\perp \cap \mathrm{span}(\Omega)$ is a subspace of positive dimension in $\Sym^d$.  By the Terracini convexity property applied to the cone $\mathcal{B}(\Sym^d_+)$, we have that $\mathrm{null}(\mathcal{B})^\perp \cap \Omega \neq \emptyset$.  Hence, there exists a proper face $\hat{\Omega}$ of $\Sym^d_+$ such that $\hat{\Omega} \subsetneq \Omega$, $\mathrm{null}(\mathcal{B})^\perp \cap \Omega = \mathrm{null}(\mathcal{B})^\perp \cap \hat{\Omega}$, and $\mathrm{null}(\mathcal{B})^\perp \cap \mathrm{ri}(\hat{\Omega}) \neq \emptyset$.  As a consequence, there also exists an element $W \in [\Omega \cap \mathrm{span}(\hat{\Omega})^\perp] \backslash \{0\}$.

We use the $W$ available to us to construct a linear map $\tilde{\mathcal{A}}$ in $\mathfrak{S}$.  Specifically, there exists $\epsilon > 0$ such that:
\begin{equation*}
\mathrm{null}(\tilde{\mathcal{A}})^\perp = \{M - \epsilon \|M\| W \;:\; M \in \mathrm{null}(\mathcal{A})^\perp \}
\end{equation*}
for some $\tilde{\mathcal{A}} \in \mathfrak{S}$.  Associated to this $\tilde{\mathcal{A}}$ is the linear map $\tilde{\mathcal{B}}$.  We show next that $\mathrm{null}(\tilde{\mathcal{B}})^\perp \cap \Omega = \{0\}$.  As $\tilde{\mathcal{B}}$ is surjective, we may consider the direct sum decomposition $\mathrm{null}(\tilde{\mathcal{B}})^\perp = \mathrm{null}(\tilde{\mathcal{A}})^\perp \oplus \mathrm{span}(I)$.  Thus, for any $Y \in \mathrm{null}(\tilde{\mathcal{B}})^\perp$, we have the decomposition $Y = M - \epsilon \|M\| W + c I$ for some $M \in \mathrm{null}(\mathcal{A})^\perp, c \in \R$.  If $Y \in \Omega$ then one can check that $Y + \epsilon \|M\| W \in \Omega$, and in particular, that $Y + \epsilon \|M\| W \notin \hat{\Omega}$ based on the construction of $W$, unless $M = 0$.  But we also have that $Y + \epsilon \|M\| W = M + c I$ and $M + cI \in \hat{\Omega}$, which implies that $M = 0$ and in turn that $c = 0$.  In summary, we obtain that $\mathrm{null}(\tilde{\mathcal{B}})^\perp \cap \Omega = \{0\}$.

	Next, we prove that $\mathrm{null}(\tilde{\mathcal{B}})^\perp \cap \mathrm{span}(\Omega)$ is a subspace of positive dimension by constructing a nonzero element in this subspace.  Recall that $\mathrm{null}(\mathcal{B})^\perp \cap \mathrm{ri}(\hat{\Omega}) \neq \emptyset$ and that $\hat{\Omega} \subset \Omega$.   Consider any $Z \in [\mathrm{null}(\mathcal{B})^\perp \cap \mathrm{ri}(\hat{\Omega})]$, which by construction is nonzero.  We have the expression $Z = M + c I$ with $M \in \mathrm{null}(\mathcal{A})^\perp \backslash \{0\}$ and $c \in \R$ based on the surjectivity of $\mathcal{B}$ and that $I \notin \Omega$.  It follows that $Z - \epsilon \|M\| W \in \mathrm{span}(\Omega) \backslash \{0\}$ as $Z \in \hat{\Omega} \backslash \{0\}$ and $W \in [\Omega \cap \mathrm{span}(\hat{\Omega})^\perp] \backslash \{0\}$.  Further, we also have that $Z - \epsilon \|M\| W = M - \epsilon \|M\| W + cI \in \mathrm{null}(\tilde{\mathcal{B}})^\perp$, as $M - \epsilon \|M\| W \in \mathrm{null}(\tilde{\mathcal{A}})^\perp$ and $\mathrm{null}(\tilde{\mathcal{B}})^\perp = \mathrm{null}(\tilde{\mathcal{A}})^\perp \oplus \mathrm{span}(I)$.  As a result, we have that $Z - \epsilon \|M\| W \in \mathrm{null}(\tilde{\mathcal{B}})^\perp \cap \mathrm{span}(\Omega) \backslash \{0\}$.

Finally, we consider the preceding two paragraphs together in the context of the Terracini convex property of the cone $\tilde{\mathcal{B}}(\Sym^d_+)$.  Specifically, we have that $\mathrm{null}(\tilde{\mathcal{B}})^\perp \cap \Omega = \{0\}$ and that $\mathrm{null}(\tilde{\mathcal{B}})^\perp \cap \mathrm{span}(\Omega)$ is a subspace of positive dimension.  This violates the reformulation of Terracini convexity of $\tilde{\mathcal{B}}(\Sym^d_+)$ that $\mathrm{span}(\mathrm{null}(\tilde{\mathcal{B}})^\perp \cap \Omega) = \mathrm{null}(\tilde{\mathcal{B}})^\perp \cap \mathrm{span}(\Omega)$.  This gives us the desired contradiction.
\end{proof}

Given the preceding two results, we now prove Theorem~\ref{thm:psd-inv}:

\begin{proof}[{Proof of Theorem~\ref{thm:psd-inv}}] For the first statement, we are given that $\mathrm{null}(\mathcal{A}) \cap \Sym^d_{++} \neq \emptyset$.  Hence, we can apply Lemma~\ref{lem:psd-inv} and obtain that the cone $\mathcal{B}(\Sym^d_+)$ satisfies the unique preimage property.  Next, in preparation to apply the first part of Proposition~\ref{prop:psd-inv}, we need to check that the linear map $\mathcal{B}$ is surjective, which is equivalent to $\mathcal{A}$ being surjective and $I \notin \mathrm{null}(\mathcal{A})^\perp$.  The former condition holds by assumption and the latter condition is in turn equivalent to $\mathrm{null}(\mathcal{A}) \nsubseteq \mathrm{span}(I)^\perp$.  The assumption $\mathrm{null}(\mathcal{A}) \cap \Sym^d_{++} \neq \emptyset$ implies that $\mathrm{null}(\mathcal{A}) \nsubseteq \mathrm{span}(I)^\perp$.  Thus, we are in a position to apply Proposition~\ref{prop:psd-inv} and obtain that the cone $\mathcal{B}(\Sym^d_+)$ satisfies the Terracini convexity property.

	For the second statement, we can apply the second part of Proposition~\ref{prop:psd-inv} to conclude that the cone $\mathcal{B}(\Sym^d_+)$ satisfies the unique preimage property.  Applying Lemma~\ref{lem:psd-inv}, we conclude that the map $\mathcal{A}$ satisfies the exact recovery property.
\end{proof}

\subsection{New Families of Terracini Convex Cones}
\label{sec:inv-new}
The results from the preceding section lead naturally to new families of Terracini convex cones.  Specifically, from the literature on the semidefinite relaxation \eqref{eq:R1} we have that the exact recovery property is satisfied with high probability by random linear maps $\mathcal{A}$ of suitable dimension \cite{rfp2010nuclear,candesplan2011}.  Combined with the first part of Theorem~\ref{thm:psd-inv}, we obtain Terracini convex cones that are specified as linear images of the cone of the positive-semidefinite matrices.

\begin{theorem}
	\label{thm:most-tc}
	Let $A_1,\dots,A_n \in \R^{d \times d}$ be a collection of independent random matrices in which each $A_i$ is a Gaussian random matrix with i.i.d entries that have zero-mean and variance $\tfrac{1}{n}$, and suppose  $n \leq (1/2-\epsilon) {d+1 \choose 2}$ for some $\epsilon\in (0,1/2)$.  Consider the linear map $\mathcal{B}: \mathbb{S}^d \rightarrow \R^{n+1}$ defined as $\mathcal{B}(X) = \begin{pmatrix}\mathrm{tr}(A_1 X) \\ \vdots \\ \mathrm{tr}(A_n X) \\ \mathrm{tr}(X)\end{pmatrix}$.  There exist constants $c_1, c_2 > 0$ and $c_3(\epsilon)>0$ (depending on $\epsilon$), such that for $k = \lfloor \tfrac{c_1n}{d} \rfloor$, the cone $\mathcal{B}(\mathbb{S}^d) \subset \R^{n+1}$ is $k$-Terracini convex with probability greater than $1-2e^{-c_2n} - e^{-c_3(\epsilon)n}$.
\end{theorem}

\begin{proof}
We begin with a geometric reformulation of the exact recovery property of Section~\ref{sec:inv-psd} based on the argument presented in Lemma~\ref{lem:psd-inv}.  Specifically, for a given linear map $\mathcal{A} : \Sym^d \rightarrow \R^n$ and a positive integer $k$, the exact recovery property of Section~\ref{sec:inv-psd} is equivalent to the condition that for any $X^\star \in \Sym^d$ with $\mathrm{rank}(X^\star) \leq k$ and $\mathrm{tr}(X^\star) = 1$, we have that $\mathcal{A}(X^\star)$ has a unique preimage in the solid spectraplex $\fs^d_0 = \{X \in \Sym^d \;:\; \mathrm{tr}(X) \leq 1, ~ X \succeq 0\}$.

	The results in \cite{rfp2010nuclear,candesplan2011} concern a more general geometric criterion which can be specialized to our context.  These results are stated in terms of the matrix nuclear norm $\|\cdot\|_\star = \sum_i \sigma_i(\cdot)$ (i.e., the sum of the singular values).  Consider the linear map $\hat{\mathcal{A}} : \R^{d \times d} \rightarrow \R^n$ defined in terms of the Gaussian random matrices $A_1,\dots,A_n$ as $\hat{\mathcal{A}}(M) = \begin{pmatrix} \mathrm{tr}(A_1 M) \\ \vdots \\ \mathrm{tr}(A_n M) \end{pmatrix}$.  There exist constants $c_1,c_2>0$ such that if $k = \lfloor \tfrac{c_1n}{d} \rfloor$, then with probability at least $1-2e^{-c_2n}$, for every $M^\star \in \R^{d \times d}$ with $\mathrm{rank}(M^\star) \leq k$ and $\|M^\star\|_\star = 1$, the point $\hat{\mathcal{A}}(M^\star)$ has a unique preimage in the nuclear norm ball $\{M \in \R^{d \times d} \;:\; \|M\|_\star \leq 1\}$ \cite{candesplan2011}.  Note that the solid spectraplex $\fs^d_0 \subset \Sym^d \subset \R^{d \times d}$ is a subset of the nuclear norm unit ball.  Thus, with the same value of $k = \lfloor \tfrac{c_1 n}{d} \rfloor$, one can conclude that the linear map $\mathcal{A}$ defined by the \emph{restriction} of $\hat{\mathcal{A}}$ to the domain $\Sym^d$
		satisfies the exact recovery property of Section~\ref{sec:inv-psd} for $k = \lfloor \tfrac{c_1 n}{d} \rfloor$ with probability greater than $1-2e^{-c_2n}$.

	Further, we have that $\mathrm{null}(\mathcal{A}) \cap \Sym^d_{++} \neq \emptyset$
	with probability at least $1-e^{-c_3(\epsilon) n}$. This follows
	from the observation that the probability
	that $\mathrm{null}(\mathcal{A}) \cap \Sym^d_{++} \neq \emptyset$ is the
	same as the probability that $\mathrm{null}(\mathcal{A}) \cap \Sym_+^d \neq \{0\}$.
	This latter quantity can be estimated using the results from \cite{gordon1988,crpw2012atomic,livingedge2014}, using the fact that the positive semidefinite
	cone is self-dual.

	Therefore, by a union bound, the assumptions of the first part of Theorem~\ref{thm:psd-inv} are satisfied,
	and hence the cone $\mathcal{B}(\Sym^d)$ is $k$-Terracini convex,
	with probability at least $1-2e^{-c_2n} - e^{-c_3(\epsilon)n}$.
\end{proof}

Thus, in some sense `most' linear images of the cone of positive semidefinite matrices are $k$-Terracini convex for a suitable $k$ depending on the dimension of the image of the linear map.  This result offers a semidefinite analog of the result of Donoho and Tanner \cite{donohotanner2005neighborliness} on neighborliness of linear images the nonnegative orthant.  Linear images of the positive semidefinite cone are semialgebraic but are generally not basic semialgebraic (as this property is not preserved under linear projections).  In the next section, we describe an approach to obtaining basic semialgebraic Terracini convex cones from the positive semidefinite cone via an different construction based on the viewpoint of hyperbolic programming.

\section{Terracini Convexity and Derivative Relaxations of Hyperbolicity Cones}
\label{sec:hyperbolic}
In this section, we study Terracini convexity from a more algebraic perspective by focusing on a class of convex cones that are obtained from hyperbolic polynomials, which are multivariate polynomials possessing certain real-rootedness properties.  The associated cones are called hyperbolicity cones, and among the prototypical examples of such cones are the nonnegative orthant and the positive semidefinite cone.  Where the previous section demonstrated that generic linear images of the nonnegative orthant and of the positive semidefinite cone are $k$-Terracini convex for suitable $k$, here we show that the (algebraically defined) operation of taking derivative relaxations of the nonnegative orthant and of the positive semidefinite cone lead to hyperbolicity cones with non-trivial Terracini convexity properties.  As hyperbolicity cones are basic semialgebraic, i.e., they are defined by finitely many polynomial inequalities, a remarkable fact about the $k$-Terracini convex cones we construct in this section is that they are all basic semialgebraic.  In contrast, the $k$-Terracini convex cones constructed in Section~\ref{sec:inverse} by taking projections of the
positive semidefinite cone are, in general, not basic semialgebraic.

The rest of the section is organized in the following way. In Section~\ref{sec:hyp-basics}, we briefly state basic definitions and terminology related to hyperbolic polynomials, hyperbolicity cones, and their derivative relaxations, as well as reviewing properties of the boundary and extreme rays of hyperbolicity cones. In Section~\ref{sec:hyp-tangent} we study tangent cones of hyperbolicity cones and how these interact with derivative relaxations.  In particular, we show that the tangent cone to a hyperbolicity cone at a point is the hyperbolicity cone associated with the localization	of the associated hyperbolic polynomial at that point. This gives us an algebraic handle on the objects arising in the definition of Terracini convexity. Section~\ref{sec:deriv-kterracini} is focused on establishing the main result on Terracini convexity properties of derivative relaxations of a class of hyperbolicity cones that includes the orthant, the positive semidefinite cone, and the cone of positive semidefinite Hankel matrices.

\subsection{Hyperbolicity Cones and Their Derivative Relaxations}
\label{sec:hyp-basics}
\paragraph{Hyperbolic polynomials}
Let $p$ be a polynomial with real coefficients that is homogeneous of degree $d$ in $n$ variables, and let $e\in \RR^n$.
We say that $p$ is \emph{hyperbolic with respect to $e$} if
$p(e)>0$ and, for each $x\in \RR^n$, the univariate polynomial $t\mapsto p(te-x)$ has only real roots.
Given $x\in \RR^n$ let
$\lambda_{\max}^{p,e}(x) = \lambda_1^{p,e}(x)\geq \lambda_2^{p,e}(x) \geq \cdots \geq \lambda_d^{p,e}(x) = \lambda_{\min}^{p,e}(x)$
 denote the roots of $t\mapsto p(te-x)$, or \emph{hyperbolic eigenvalues} of $x$ with respect to $p$ and $e$.  If $p$ and $e$ are clear from the context, we write $\lambda_1(x), \cdots, \lambda_d(x)$.
The \emph{rank} of $x\in \RR^n$, denoted $\textup{rank}_{p}(x)$, is the number of non-zero hyperbolic
eigenvalues of $x$ with respect to $p$ and $e$. The \emph{multiplicity} of $x$
is $\textup{mult}_p(x) = \textup{deg}(p) - \textup{rank}_p(x)$,
the number of zero hyperbolic eigenvalues of $x$.

\paragraph{Hyperbolicity cones}
Associated with a hyperbolic polynomial $p$ and direction of hyperbolicity $e$ is the \emph{closed hyperbolicity cone}
$\Lambda_+(p,e) = \{x\in \RR^n\;:\;\lambda_{\min}^{p,e}(x) \geq 0\}$.
This is a convex cone, a result due to G\r{a}rding~\cite{gaarding1959inequality}.  We denote the interior of this cone by $\Lambda_{++}(p,e)$.  If $\tilde{e} \in \Lambda_{++}(p,e)$, then $p$ is hyperbolic with respect $\tilde{e}$ and $\Lambda_+(p,e) = \Lambda_+(p,\tilde{e})$~\cite{gaarding1959inequality}.
If $p$ and $e$ are clear from the context, we write $\Lambda_+$ instead of $\Lambda_+(p,e)$ for brevity of notation.

Although the hyperbolic eigenvalues of $x$ with respect to $p$ depend on the choice of $e$,
the multiplicity, $\textup{mult}_p(x)$, and rank, $\textup{rank}_p(x)$, are independent of the choice of direction of
hyperbolicity~\cite[Proposition 22]{renegar2006hyperbolic}.
The lineality space of the hyperbolicity cone $\Lambda_+$ is exactly the set of points with
multiplicity $\textup{deg}(p)$ (or rank zero), i.e.,
\begin{equation}
\label{eq:lin-rank} \Lambda_+\cap (-\Lambda_+) = \{x\in \RR^n\;:\; \textup{mult}_p(x) = \textup{deg}(p)\}.
\end{equation}
(see, e.g.,~\cite[Proposition 11]{renegar2006hyperbolic}).
If we expand $p(x+te)$ in powers of $t$ as
\begin{equation}
\label{eq:hyp-coefs}
	 p(x+te) = a_0t^d+a_1(x)t^{d-1} + \cdots + a_{d-2}(x)t^2 + a_{d-1}(x)t + a_d(x),
\end{equation}
then Descartes' rule of signs gives an equivalent
description of the hyperboicity cone as
\[ \Lambda_+(p,e) = \{x\in \RR^n\;:\; a_d(x)\geq 0,\; a_{d-1}(x)\geq 0,\;a_{d-2}(x)\geq 0,\;\ldots,\; a_1(x) \geq 0\}.\]
This shows that any hyperbolicity cone is a \emph{basic semialgebraic} set, i.e., it can be expressed
via finitely many polynomial inequalities.

\paragraph{Derivative relaxations}
If $p$ is hyperbolic with respect to $e$ and $\tilde{e} \in \Lambda_{++}(p,e)$, then the directional derivative
\[ D_{\tilde{e}}p(x):= \left.\frac{d}{dt}p(x+t \tilde{e}) \right|_{t=0}\]
is again hyperbolic with respect to $e$ (by Rolle's theorem).
The hyperbolicity cone $\Lambda_+(D_{\tilde{e}}p,e)$ satisfies $\Lambda_+(D_{\tilde{e}}p,e) \supseteq \Lambda_+(p,e)$. As such, it is
often referred to as a \emph{derivative relaxation} of $\Lambda_+(p,e)$. When $p$, $e$, and $\tilde{e}$ are clear from
the context we abuse notation and write $\Lambda_+':= \Lambda_+(D_{\tilde{e}}p,e)$ for brevity.

One of the most interesting aspects of derivative relaxations is that boundary points (of high enough multiplicity) of $\Lambda_+$
remain boundary points of $\Lambda_+'$.
\begin{theorem}[{Renegar~\cite[Theorem 12]{renegar2006hyperbolic}}]
\label{thm:mult3}
Let $p$ be hyperbolic with respect to $e$ with hyperbolicity cone $\Lambda_+$, and for any $\tilde{e} \in \Lambda_{++}(p,e)$ let the associated derivative relaxation be $\Lambda_+'$.
If $m\geq 3$ then
	\[ \{x\in \Lambda_+\;:\; \textup{mult}_p(x) = m\} = \{x\in \Lambda_+'\;:\; \textup{mult}_{D_{\tilde{e}}p}(x) = m-1\}.\]
\end{theorem}
As a straightforward corollary, we obtain a relationship between the lineality spaces of a hyperbolicity cone and
its derivative relaxation.
\begin{corollary}
\label{cor:lin-d3}
Under the same hypotheses as Theorem~\ref{thm:mult3}, if $\textup{deg}(p) \geq 3$ then $\Lambda_+\cap (-\Lambda_+)= \Lambda_+' \cap (-\Lambda_+')$.
\end{corollary}
\begin{proof}
This follows from Theorem~\ref{thm:mult3} by noting that the lineality space of $\Lambda_+$ is exactly
the set of $x$ with $\textup{mult}_p(x)=\textup{deg}(p)$ and the lineality space of $\Lambda_+'$ is exactly the
set of $x$ with $\textup{mult}_{D_{\tilde{e}}p}(x) = \textup{deg}(p)-1$.
\end{proof}
One consequence of Corollary~\ref{cor:lin-d3} is that if $\textup{deg}(p) \geq 3$ then $\Lambda_+$ being a pointed cone implies that any derivative relaxation $\Lambda_+'$ is also pointed.  Building on Corollary~\ref{cor:lin-d3}, we can understand how the extreme rays of the derivative cone and the original cone relate to each other.
In particular, the extreme rays of derivative relaxations are either extreme rays of the original cone or
extreme rays of multiplicity one.
\begin{corollary}
	\label{cor:ext-der}
Assume that $\Lambda_+$ is pointed and $\textup{deg}(p) \geq 3$, and let $\Lambda_+'$ be the derivative relaxation associated to any $\tilde{e} \in \Lambda_{++}(p,e)$. If $x$ generates an extreme ray of $\Lambda_+'$ then either $\textup{mult}_{D_{\tilde{e}}p}(x) = 1$ or $x$ generates an extreme ray of $\Lambda_+$ and $\textup{mult}_p(x) \geq 3$.
\end{corollary}
\begin{proof}
As $\Lambda_+$ is pointed and $\textup{deg}(p)\geq 3$ it follows from Corollary~\ref{cor:lin-d3}
that $\Lambda_+'$ is pointed. If $x$ generates an extreme ray of $\Lambda_+'$ and $\textup{mult}_{D_{\tilde{e}}p}(x) \geq 2$ then,
by Theorem~\ref{thm:mult3}, we can conclude that $\textup{mult}_{p}(x) \geq  3$ and $x\in \Lambda_+$.
Since $x\in \Lambda_+ \supseteq \Lambda_+'$ and $x$ generates an extreme ray of $\Lambda'$, it follows that $x$ generates an extreme ray of $\Lambda_+$.
\end{proof}

\subsection{Tangent Cones and Derivative Relaxations}
\label{sec:hyp-tangent}
In this section we study tangent cones of hyperbolicity cones, and in particular how tangent cones change when we take derivative relaxations.  We first show that the
tangent cone of a hyperbolicity cone $\Lambda_+(p,e)$ at a point $x$ is again a
hyperbolicity cone (Theorem~\ref{thm:hyp-tangent}) and that the corresponding
hyperbolic polynomial is the localization of $p$ at $x$
(Definition~\ref{def:hyp-localization}).  The main result of the section
(Theorem~\ref{thm:tangent-derivative}) is that the
tangent cone to $\Lambda_+'$ at a boundary point $x$
is the corresponding derivative relaxation of the
tangent cone to $\Lambda_+$ at that same point $x$.
This is the key technical result that enables us to
understand how $k$-Terracini convexity is affected by taking derivative
relaxations (see Section~\ref{sec:deriv-kterracini}).

\begin{definition}
\label{def:hyp-localization}
If $p$ is a hyperbolic polynomial with respect to $e$ and with associated hyperbolicity cone $\Lambda_+$, then the
\emph{localization of $p$ at $x\in \Lambda_+$}
is the polynomial of degree $\textup{mult}_p(x)$ defined by
\[ \loc{p}{x}(y) = \lim_{\lambda\rightarrow \infty}\lambda^{\textup{mult}_p(x)}p(x + \lambda^{-1} y) =
\lim_{\lambda\rightarrow\infty}\lambda^{-\textup{rank}_p(x)}p(\lambda x + y).\]
\end{definition}
\begin{example}
\label{eg:psd-loc}
Let $p(X) = \det(X)$ where $X$ is a $d\times d$ symmetric matrix of indeterminates, and let $e = I$. The corresponding
hyperbolicity cone is the cone of $d\times d$ positive semidefinite matrices.
Suppose that $X = \left[\begin{smallmatrix} Z & 0\\0 & 0\end{smallmatrix}\right]$ where
$Z$ is $k\times k$ and positive definite. Then, by the formula for the determinant of a block matrix in terms of the Schur complement,
\begin{align*}
	\loc{p}{X}\left(\begin{bmatrix} Y_{11}&Y_{12}\\Y_{12}^T & Y_{22}\end{bmatrix}\right)
	& = \lim_{\lambda\rightarrow {\infty}} \lambda^{d-k}
	\det\left(
	\begin{bmatrix} Z + \lambda^{-1}Y_{11} & \lambda^{-1}Y_{12}\\\lambda^{-1}Y_{12}^T & \lambda^{-1}Y_{22}\end{bmatrix}
	\right)\\
	& = \lim_{\lambda\rightarrow \infty}\lambda^{d-k}\det(\lambda^{-1}Y_{22}) \det(Z + \lambda^{-1}Y_{11} -
	\lambda^{-1}Y_{12}Y_{22}^{-1}Y_{12}^T)\\
	& = \det(Z)\det(Y_{22}).
\end{align*}
\end{example}
There is an alternative formulation of $\loc{p}{x}$ in terms of directional derivatives of $p$ in the $x$ direction.
This alternative formulation is particularly useful in understanding how derivative relaxations interact with localization.  In the forthcoming discussion, we refer on several occasions to higher-order directional derivatives of a hyperbolic polynomial, which we denote as a composition of first-order directional derivatives as $D_{y^{(k)}} \cdots D_{y^{(1)}} p$; if the directions $y^{(1)},\dots,y^{(k)}$ are the same, we denote the associated higher-order directional derivative in a more compact manner as $D_{y}^k p$.
\begin{lemma}
\label{lem:loc-alt}
If $p$ is a hyperbolic polynomial with respect to $e$ then
\[ \loc{p}{x}(y) = \frac{1}{\textup{mult}_{p}(x){!}}D^{\textup{mult}_p(x)}_y p(x).\]
\end{lemma}
\begin{proof}
By a Taylor expansion,
\[ p(x+\lambda^{-1}y) = \sum_{k=0}^{\textup{deg}(p)}\frac{\lambda^{-k}}{k{!}}{D^{k}_y p(x)}.\]
Since $p$ vanishes to order $\textup{mult}_{p}(x)$ as $\lambda\rightarrow \infty$,
and $\textup{mult}_p(x)$ is independent of the choice of $e$ in the interior of the hyperbolicity cone, it follows that
$D^{k}_y p(x)=0$ whenever $y\in\textup{int}(\Lambda_{+}(p,e))$ and $0\leq k < \textup{mult}_p(x)$. As such,
if $k<\textup{mult}_p(x)$ then $y\mapsto D^k_y p(x)$ is a polynomial that vanishes on the interior of
the (full-dimensional) hyperbolicity cone, so it must be identically zero. Hence
\[ \lambda^{\textup{mult}_p(x)}p(x+\lambda^{-1}y) =
\sum_{k=\textup{mult}_p(x)}^{\textup{deg}(p)}\frac{\lambda^{\textup{mult}_p(x)-k}}{k{!}} D^k_y p(x).\]
Taking the limit as $\lambda\rightarrow \infty$ we obtain the stated result.
\end{proof}
We now consider localization of a hyperbolic polynomial from the point of view of its zeros. To do so, we use the following basic fact
about how hyperbolic eigenvalues change along different directions.
\begin{lemma}[{\cite[Lemma 3.27]{atiyah1970lacunas}}]
\label{lem:abg}
Suppose $p$ is hyperbolic with respect to $e$. If $x,u\in \RR^n$ then
\[ p(x-te+su) = p(e)\prod_{i=1}^{{\textup{deg}(p)}}(t_i(s;x,u)-t)\]
where the functions $s\mapsto t_i(s;x,u)$ are real analytic functions of $s$. Furthermore, if $u\in \Lambda_+(p,e)$ then
$t'_{i}(s;x,u) := \frac{d}{ds}t_i(s;x,u) \geq 0$ for all $s$.
\end{lemma}
The roots of the polynomial $t \mapsto p(x-te+su)$ are the eigenvalues of $x+su$, and therefore each $t_i(s;x,u)$ in the above lemma is an eigenvalue of $x+su$.  The assertion that the functions $s\mapsto t_i(s;x,u)$ are real analytic functions of $s$ corresponds to the eigenvalues of $x+su$ being analytic functions of $s$, and the nonnegativity of each of the derivatives $t'_{i}(s;x,u)$ (when $u \in \Lambda_+(p,e)$) corresponds to each of the eigenvalues of $x+su$ being non-decreasing functions of $s$.
This result is useful because it allows us to understand localization from the point of view of eigenvalues.
\begin{lemma}
\label{lem:loc-eig}
Suppose $p$ is hyperbolic with respect to $e$ and fix some $x \in \Lambda_+(p,e)$.  Letting $m = \textup{mult}_p(x)$ we have that
\begin{equation}
	\label{eq:loc-eig}
	\loc{p}{x}(y-te) = p(e)\prod_{i=1}^{\textup{deg}(p)-m} \lambda_i(x) \prod_{j=\textup{deg}(p)-m+1}^{\textup{deg}(p)} (t'_j(0;x,y) - t).
\end{equation}
\end{lemma}
\begin{proof}
If $x \in \Lambda_+(p,e)$ has multiplicity $m := \textup{mult}_{p}(x)$ then the functions $t_i$ in the factorization of
Lemma~\ref{lem:abg} have the property that $t_{1}(0;x,u) = \cdots = t_{\textup{deg}(p)-m}(0;x,u) > 0$ and that $t_{\textup{deg}(p)-m+1}(0;x,u), \ldots, t_{\textup{deg}(p)}(0;x,u) \allowbreak = 0$ by virtue of $t_i(0;x,u)$ being eigenvalues of $x$. Using the factorization of Lemma~\ref{lem:abg}, we see that
\begin{multline}
	 \lambda^{\textup{mult}_p(x)}p(x+\lambda^{-1}(y - te)) =\\
	p(e) \prod_{i=1}^{\textup{deg}(p)-m} (t_{i}(\lambda^{-1};x,y) - \lambda^{-1}t) \prod_{j=\textup{deg}(p)-m+1}^{\textup{deg}(p)}(\lambda t_j(\lambda^{-1};x,y)-t).
	\label{eq:pre-loc-eig}
\end{multline}
Expanding $t_i(\lambda^{-1};x,y)$ about $t_i(0;x,y)$, gives $t_i(\lambda^{-1};x,y) = \lambda_i(x) + \lambda^{-1}t'(0;x,y) + O(\lambda^{-2})$. We obtain~\eqref{eq:loc-eig} by taking
the limit as $\lambda\rightarrow \infty$.
\end{proof}
We are interested in the localization of a hyperbolic polynomial at a point
because it turns out to be the algebraic analogue of the geometric operation of taking the
tangent cone to a hyperbolicity cone at a point. Although this is probably well-known, we have included a proof
because we had difficulty finding an explicit statement of this type in the literature.

We now show that localization at $x$ is the algebraic analog of the tangent cone to the hyperbolicity cone at $x$.
\begin{theorem}
\label{thm:hyp-tangent}
If $p$ is hyperbolic with respect to $e$ and $x\in \Lambda_+(p,e)$ then
\begin{enumerate}
	\item $\loc{p}{x}$ is hyperbolic with respect to $e$; and
	\item $\Lambda_+(\loc{p}{x},e) = \overline{\cfd_{\Lambda_+(p,e)}(x)}$ is the tangent cone of $\Lambda_+(p,e)$ at $x$.
\end{enumerate}
\end{theorem}
\begin{proof}
The fact that the localization is hyperbolic with respect to $e$ is exactly~\cite[Lemma 3.42]{atiyah1970lacunas},
and also follows immediately from Lemma~\ref{lem:loc-eig} and the fact that the $t_i'(0;x,y)$ are always real.

For the second part, we first show that the hyperbolicity cone of the localization at $x$ is contained in the tangent cone of $\Lambda_+(p,e)$ at $x$. Let $z\in \Lambda_{++}(\loc{p}{x},e)$ be in the interior of the hyperbolicity cone of the localization at $x$. Then,
from~\eqref{eq:loc-eig} we know that $t_{j}'(0;x,z) > 0$ for $j = \textup{deg}(p) - \textup{mult}_p(x) + 1, \ldots, \textup{deg}(p)$. Furthermore, since $x\in \Lambda_+(p,e)$
we know that $t_{i}(0;x,z) = \lambda_i(x)>0$ for $i=1, \ldots, \textup{mult}_p(x)$. From~\eqref{eq:pre-loc-eig} we know that
the roots of $t\mapsto \lambda^{\textup{mult}_p(x)}p(x+\lambda^{-1}(z - te))$ are $\lambda t_i(\lambda^{-1};x,z) = \lambda \lambda_i(x) + t_i'(0;x,z) + O(\lambda^{-1})$ for $i=1, \ldots, \textup{deg}(p)$. As such, there exists a sufficiently large positive $\lambda_0$ such that if
$\lambda\geq \lambda_0$ then all of these roots are positive.
Hence we have that $x+z/\lambda_0 \in \Lambda_{++}(p,e)$ and therefore $z\in \cfd_{\Lambda_+(p,e)}(x)$, the cone of feasible directions with respect to $\Lambda_+(p,e)$ at $x$.
We have shown that $\Lambda_{++}(\loc{p}{x},e)\subseteq \cfd_{\Lambda_+(p,e)}(x)$. Taking closures shows that
$\Lambda_+(\loc{p}{x},e)$ is contained in the tangent cone of $\Lambda_+(p,e)$ at $x$.

For the reverse inclusion, suppose that $z\in \cfd_{\Lambda_+(p,e)}(x)$. In other words, there exists a
sufficiently large positive $\lambda_0$ such that $x+\lambda^{-1}z \in \Lambda_+(p,e)$ for all $\lambda \geq \lambda_0$. Then
\[ t\mapsto \lambda^{\textup{mult}_{p}(x)}p(x+\lambda^{-1}(z+te))\]
has nonnegative coefficients for all $\lambda \geq \lambda_0$. By continuity of the coefficients as functions of $\lambda$,
it follows that
\[ t\mapsto \lim_{\lambda\rightarrow \infty}\lambda^{\textup{mult}_p(x)}p(x+\lambda^{-1}(z+te)) = \loc{p}{x}(z+te)\]
has non-negative coefficients. Consequently, we have that $z\in \Lambda_+(\loc{p}{x},e)$ and so the cone of feasible directions is
contained in $\Lambda_+(\loc{p}{x},e)$. Taking the closure shows that the tangent cone is contained in $\Lambda_+(\loc{p}{x},e)$,
completing the proof.
\end{proof}
\begin{example}[{Example~\ref{eg:psd-loc} continued}]
\label{eg:psd-tangent-hyp}
Suppose that $p(X) = \det(X)$ where $X$ is a $d\times d$ symmetric matrix of indeterminates, $e = I$, and
$X = \left[\begin{smallmatrix} Z & 0\\0 & 0\end{smallmatrix}\right]$ where
$Z$ is $k\times k$ and positive definite. The hyperbolicity cone of $\loc{p}{X}$ is
\[ \Lambda_+(\loc{p}{X},I) = \left\{\begin{bmatrix}Y_{11} & Y_{12}\\Y_{12}^T & Y_{22}\end{bmatrix} \;:\; Y_{22} \psd 0\right\}\]
which coincides with the tangent cone to the positive semidefinite cone at $X$.
\end{example}
The main technical
result of this section, a fairly immediate corollary of Lemma~\ref{lem:loc-alt},
is that localization and taking derivatives commute.
\begin{theorem}
	\label{thm:tangent-derivative}
	If $p$ is hyperbolic with respect to $e$ and $x \in \Lambda_+(p,e)$ with $\textup{mult}_p(x) \geq 1$, then $\Lambda_+(D_{\tilde{e}}\loc{p}{x},e) = \Lambda_+(\loc{D_{\tilde{e}}p}{x},e)$ for any $\tilde{e} \in \Lambda_{++}(p,e)$.
\end{theorem}
\begin{proof}
	Let $m = \textup{mult}_p(x) \geq 1$.
	On the one hand $\loc{p}{x}(y) = \frac{1}{{m}{!}} {D_y^m} p(x)$. Differentiating in the direction ${\tilde{e}}$ gives
\begin{equation*}
\begin{aligned}
D_{\tilde{e}}\loc{p}{x}(y) &= \left.\frac{d}{dt}\frac{1}{{m}{!}} {D_{y+t\tilde{e}}^m} p(x)\right|_{t=0} \\ &= { \frac{d}{dt}\frac{1}{m{!}} \left[\sum_{i=0}^m {m \choose i} t^i D^i_{\tilde{e}} D^{m-i}_y p(x) \right] \Bigg|_{t=0}} \\ &= \frac{1}{({m}-1){!}} {D_{\tilde{e}}} {D^{m-1}_y} p(x).
\end{aligned}
\end{equation*}
	where we have used the fact that ${D_{y^{(1)}} \cdots D_{y^{(m)}}} p(x)$ is invariant under permutations of {$y^{(1)},\ldots, y^{(m)}$}.
	On the other hand $x$ has multiplicity $m-1\geq 0$ with respect to $D_{\tilde{e}}p$. As such
	\[ \loc{D_{\tilde{e}}p}{x}(y) = \frac{1}{({m}-1){!}} {D^{m-1}_y} D_{\tilde{e}} p(x).\]	
	We have shown that $D_{\tilde{e}}\loc{p}{x} = \loc{D_{\tilde{e}} p}{x}$, from which the result directly follows.
\end{proof}
The fact that localization and taking derivatives commute tells us that {the convex tangent space of a hyperbolicity cone is exactly the same as the convex tangent space of its derivative relaxation at points of high enough multiplicity.}
\begin{corollary}
	\label{cor:cl-deriv}
	If $p$ is hyperbolic with respect to $e$, $x\in \Lambda_+$, and $\textup{mult}_p(x) \geq 3$, then
	$\cl_{\Lambda_+'}(x) = \cl_{\Lambda_+}(x)$ {for any derivative relaxation $\Lambda_+' = \Lambda_+(D_{\tilde{e}}p,e)$ for $\tilde{e} \in \Lambda_{++}(p,e)$}.
\end{corollary}
\begin{proof}
From Theorem~\ref{thm:hyp-tangent}, the convex tangent space of $\Lambda_+$ at $x$ is the lineality space of $\Lambda_+(\loc{p}{x},e)$.
Similarly, the convex tangent space of $\Lambda_+'$ at $x$ is the lineality space of $\Lambda_+(\loc{D_{\tilde{e}}p}{x},e)$,
which is the lineality space of $\Lambda_+(D_{\tilde{e}}\loc{p}{x},e)$ from Theorem~\ref{thm:tangent-derivative}.
Since $\textup{deg}(\loc{p}{x}) = \textup{mult}_p(x) \geq 3$ Corollary~\ref{cor:lin-d3} tells us that the lineality space of
$\Lambda_+(\loc{p}{x},e)$ is equal to the lineality space of $\Lambda_+(D_{\tilde{e}}\loc{p}{x},e)$, completing the proof.
\end{proof}

\subsection{Derivative Relaxations of Terracini Convex Hyperbolicity Cones}
\label{sec:deriv-kterracini}
In this section we state and prove two results related to Terracini convexity properties of derivative relaxations
of hyperbolicity cones. The first, Proposition~\ref{prop:dt}, gives a
sufficient condition under which any derivative relaxation
of a  hyperbolicity cone that is $k$-Terracini convex also has non-trivial Terracini convexity properties. It is,
\emph{a priori}, unclear whether the hypotheses of Proposition~\ref{prop:dt} hold for any interesting examples.
In the main result of this section (Theorem~\ref{thm:rk1-dr}), we show that if $\Lambda_+$ is a hyperbolicity cone
that is Terracini convex and for which all of its extreme rays have hyperbolic rank one, then repeatedly taking derivative
relaxations produces new examples of hyperbolicity cones with Terracini convexity properties. Examples of hyperbolicity
cones to which Theorem~\ref{thm:rk1-dr} applies are the nonnegative orthant and the positive semidefinite cone, as well as
other examples such as the cone of $d\times d$ positive semidefinite Hankel matrices.

\begin{proposition}
\label{prop:dt}
Suppose that $p$ {is hyperbolic with respect to $e$, the degree $\textup{deg}(p) \geq 3$}, and the associated {hyperbolicity} cone
$\Lambda_+$ is {pointed and} $k$-Terracini convex.
{If each collection} $x^{(1)},\ldots,x^{(k')}$ of $k'$ extreme rays of $\Lambda_+$ {satisfies one of the following conditions}:
\begin{itemize}
	\item there exists $j$ such that $\textup{mult}_p(x^{(j)}) \leq 2$, or
	\item $\textup{mult}_{p}\left(\sum_{i=1}^{k'}x^{(i)}\right) \geq 3$,
\end{itemize}
then {any derivative relaxation $\Lambda_+' = \Lambda_+(D_{\tilde{e}}p,e)$ for $\tilde{e} \in \Lambda_{++}(p,e)$} is $\min\{k,k'\}$-Terracini convex.
\end{proposition}

{\emph{Remarks}: The case $\textup{deg}(p) = 1$ is vacuous as $\Lambda_+$ is a halfspace and Terracini convexity requires a cone to be pointed.  For similar reasons, the case $\textup{deg}(p) = 2$ is not interesting as $\textup{deg}(D_{\tilde{e}} p) = 1$ and $\Lambda_+'$ is a halfspace.}

\begin{proof}
Let $\ell = \min\{k,k'\}$ and let $x^{(1)},\ldots,x^{(\ell)}$ be extreme
rays of $\Lambda_+'$ {(note that $\Lambda_+'$ is pointed as $\Lambda_+$ is pointed and $\textup{deg}(p) \geq 3$)}.  {We consider next two cases based on the multiplicities of the $x^{(j)}$'s with respect to the derivative polynomial $D_{\tilde{e}} p$.}
\paragraph{Case 1:} Assume that there exists $j$ such that $\textup{mult}_{D_{\tilde{e}}p}(x^{(j)}) = 1$. In this case,
the localization of $D_{\tilde{e}}p$ at $x^{(j)}$ has degree one, {which implies that $\Lambda_+(\loc{D_{\tilde{e}} p}{x}, e)$ is a halfspace; therefore, from the second part of Theorem~\ref{thm:hyp-tangent}, the convex tangent space of $\Lambda_+'$ at $x^{(j)}$ is a subspace of codimension one.}
If all of $x^{(1)},\ldots,x^{(\ell)}$ generate the same
extreme ray, then so does $\sum_{i=1}^{\ell}x^{(i)}$. This means that all of the convex
tangent spaces {of $\Lambda_+'$} at these points are the same, so certainly
$\cl_{\Lambda_+'} \left( \sum_{i=1}^{\ell}x^{(i)} \right) = \sum_{i=1}^{\ell}\cl_{\Lambda_+'}(x)$.
Otherwise there is some $x^{(j')}$ that generates an extreme ray that is distinct from $x^{(j)}$. Since
	$\cl_{\Lambda_+'}(x^{(j)})\cap \Lambda_+'$ exposes the extreme ray generated by $x^{(j)}$ ({from Lemma~\ref{lem:face-cl-relation}, as hyperbolicity cones are facially exposed~\cite[Theorem 23]{renegar2006hyperbolic}}),
it follows that $x^{(j')}\notin \cl_{\Lambda_+'}(x^{(j)})$. Since the convex tangent space of $\Lambda_+'$ at
$x^{(j)}$ has codimension one and does not contain $x^{(j')}$,
\begin{align*}
	\sum_{i=1}^{\ell}\cl_{\Lambda_+'}(x^{(i)})
	& \supseteq \cl_{\Lambda_+'}(x^{(j')}) + \cl_{\Lambda_+'}(x^{(j)})\\
	& \supseteq \textup{span}(x^{(j')}) + \cl_{\Lambda_+'}(x^{(j)})
	 = \RR^n \supseteq \cl_{\Lambda_+'}\left(\sum_{i=1}^{\ell}x^{(i)}\right).
\end{align*}
\paragraph{Case 2:} Assume that  $\textup{mult}_{D_{\tilde{e}}p}(x^{(i)}) \geq 2$ for all $i=1,2,\ldots,\ell$.  {From Corollary~\ref{cor:ext-der}, it} follows
that $\textup{mult}_p(x^{(i)}) \geq 3$ for all $i=1,2,\ldots,\ell$ and that the $x^{(i)}$ all generate
extreme rays of $\Lambda_+$.  {As $\ell \leq k'$} and by our assumption on the extreme rays of $\Lambda_+$,
it follows that $\textup{mult}_{p}\left(\sum_{i=1}^{\ell}x^{(i)}\right) \geq 3$.
Then
\begin{equation}
	\label{eq:pf-dh}
	\cl_{\Lambda_+'}\left(\sum_{i=1}^{\ell}x^{(i)}\right) = \cl_{\Lambda_+}\left(\sum_{i=1}^{\ell}x^{(i)}\right) =
	\sum_{i=1}^{\ell}\cl_{\Lambda_+}(x^{(i)}) = \sum_{i=1}^{\ell}\cl_{\Lambda_+'}(x^{(i)}).
\end{equation}
The first and third equalities in~\eqref{eq:pf-dh} follow from Corollary~\ref{cor:cl-deriv} together with
the fact that $x^{(i)}$ (for each $i$) and $\sum_{i=1}^{\ell}x^{(i)}$ have multiplicity
at least three with respect to $p$. The second equality in~\eqref{eq:pf-dh} follows from the fact
that ${\Lambda_+}$ is $k$-Terracini convex and ${\ell \leq k}$.
\end{proof}
While Proposition~\ref{prop:dt} may appear rather technical, it is useful because it applies when we repeatedly
take derivative relaxations.  {Indeed, we have as an immediate consequence that for a hyperbolic polynomial $p$ with $\textup{deg}(p) = 3$, if $\Lambda_+$ is Terracini convex then so is any derivative relaxation $\Lambda_+'$; this follows from the observation that multiplicity of any generator of an extreme ray of $\Lambda_+$ is at most two.  For higher-degree hyperbolic polynomials, we present next the main result of this section which shows that for Terracini convex
hyperbolicity cones with all the extreme rays having hyperbolic rank one, the derivative relaxations yield new hyperbolicity cones with non-trivial Terracini convexity properties.}
Recall that the rank of a point with respect to a hyperbolic polynomial is the 
number of non-zero eigenvalues, and that a cone is Terracini convex if it is $k$-Terracini convex
for all $k$.
\begin{theorem}
\label{thm:rk1-dr}
Let $p$ be hyperbolic with respect to $e$ and let $d = \textup{deg}(p)$ {with $d > 3$}. Suppose that $\Lambda_+(p,e)$ is {pointed and} Terracini convex and
that whenever $x$ {generates} an extreme ray of $\Lambda_+(p,e)$ then $\textup{rank}_p(x) = 1$.
If $1\leq \ell \leq d-3$ is a positive integer and ${e^{(1)},\ldots,e^{(\ell)}} \in \Lambda_{++}(p,e)$ then
$\Lambda_+({D_{e^{(\ell)}}D_{e^{(\ell-1)}}\cdots D_{e^{(1)}}}p,e)$ is $(d-\ell-2)$-Terracini convex.
\end{theorem}
\begin{proof}
For brevity of notation, {we write $p^{(i)} = D_{e^{(i)}}\cdots D_{e^{(1)}}p$ and $\Lambda^{(i)}_+ = \Lambda_+(p^{(i)},e)$ for $i=1,\dots,\ell$.}

{It is helpful in our proof to use the observation that any $x$ that generates an extreme ray of $\Lambda^{(\ell)}$ either generates an extreme ray of $\Lambda_+ := \Lambda_+(p,e)$ or satisfies $\textup{mult}_{p^{(\ell)}}(x) = 1$.  We show both this secondary result as well as the primary result via induction.}

{For the base case of the secondary result, note that if $x$ generates an extreme ray of $\Lambda_+^{(1)}$ then by Corollary~\ref{cor:ext-der} either $\textup{mult}_{p^{(1)}}(x)=1$ or $x$ is an extreme ray of $\Lambda_+$ (with $\textup{mult}_{p}(x)\geq 3$).  For the primary result, note that $\Lambda_+$ is Terracini convex. If $x^{(1)},\ldots,x^{(d-3)}$ are extreme rays of $\Lambda_+$, then their sum has rank at most $d-3$ (since the hyperbolic rank function is subadditive~\cite{amini2018non})
and hence has multiplicity at least three with respect to $p$. It follows from Proposition~\ref{prop:dt} that $\Lambda_+^{(1)}$ is $(d-3)$-Terracini convex.}

{For the inductive hypothesis of the secondary result, assume that if $x$ generates an extreme ray of $\Lambda_+^{(\ell-1)}$ then either $x$ generates an extreme ray of $\Lambda_+$ or
$\textup{mult}_{p^{(\ell-1)}}(x) = 1$.  For the primary result, assume that $\Lambda_+^{(\ell-1)}$ is $(d-\ell-1)$-Terracini convex.}

We now establish the inductive step for the {secondary} result.  If $x$ {generates} an extreme ray of $\Lambda_+^{(\ell)}$ then by Corollary~\ref{cor:ext-der} either $\textup{mult}_{p^{(\ell)}}(x)=1$ or $x$ {generates} an extreme ray of $\Lambda_+^{(\ell-1)}$ with $\textup{mult}_{p^{(\ell-1)}}(x)\geq 3$, and so by the inductive hypothesis is an extreme ray of $\Lambda_+$.

Finally we establish the inductive step of the {primary result} by applying Proposition~\ref{prop:dt}. Let
$x^{(1)},\ldots,x^{(d-\ell-2)}$ be extreme rays of $\Lambda_+^{(\ell-1)}$.
Assume that each $x^{(i)}$ has multiplicity at least three with respect to
{$p^{(\ell-1)}$} (otherwise we are done).
Based on the inductive hypothesis, each $x^{(i)}$ must be an extreme ray of $\Lambda_+$ and so must have rank one with respect to $p$ by assumption.
Then $\sum_{i=1}^{d-\ell-2}x^{(i)}$ has rank at most $d-\ell-2$, and hence multiplicity at least $\ell+2$, with respect to $p$.
By applying Theorem~\ref{thm:mult3} $\ell-1$ times, and noting that {$p, p^{(1)}, \ldots, p^{(\ell-1)}$} all have degree at least three,
we see that $\sum_{i=1}^{d-\ell-2}x^{(i)}$ has multiplicity at least $\ell+2 - (\ell-1) = 3$ with respect to {$p^{(\ell-1)}$}.
Then, by Proposition~\ref{prop:dt},  $\Lambda_+^{(\ell)}$ is $(d-\ell-2)$-Terracini convex.
\end{proof}
We conclude by discussing three concrete special cases of Theorem~\ref{thm:rk1-dr}.
\begin{example}[Hyperbolicity cones associated with permanents]
	If $p(x) = \prod_{i=1}^{d}x_i$ and $e$ is the vector of all ones,
	then the corresponding hyperbolicity cone is the nonnegative orthant.
	This is Terracini convex and all of its extreme rays have rank one.  {In this case, if $e^{(1)}, \dots, e^{(\ell)} \in \RR^d_{++}$, then $D_{e^{(\ell)}} \cdots D_{e^{(1)}} p(x)$ is the \emph{permanent} of the $d \times d$ matrix with columns $e^{(1)}, \dots, e^{(\ell)}$ and $d-\ell$ copies of $x$.  Theorem~\ref{thm:rk1-dr} then tells us that the hyperbolicity cone associated with this permanent is $(d-\ell-2)$-Terracini convex as long as $1\leq \ell \leq d-3$.}
\end{example}
\begin{example}[Hyperbolicity cones associated with mixed discriminants]
	If $p(X) = \det(X)$ and $e$ is the identity matrix, then the corresponding hyperbolicity cone
	is the positive semidefinite cone. This is Terracini convex and all of its extreme rays have rank one.
	In this case, if ${E^{(1)}},\ldots,{E^{(\ell)}}$ are positive definite matrices then the quantity
	$D_{{E^{(\ell)}}} \cdots D_{{E^{(1)}}} p(X)$ is known as the \emph{mixed discriminant} of
	the $d$-tuple of matrices $({E^{(1)}}, \ldots, {E^{(\ell)}},X,\ldots,X)$. Theorem~\ref{thm:rk1-dr} then tells us that the
	hyperbolicity cone associated with this mixed discriminant is $(d-\ell-2)$-Terracini convex as long as $1\leq \ell \leq d-3$.
\end{example}
\begin{example}[Hyperbolicity cones associated with mixed discriminants of Hankel matrices]
	Consider the cone ${\mathcal{H}}_{d+1}$ of $(d+1)\times (d+1)$
	symmetric positive semidefinite Hankel matrices. This can be viewed as the hyperbolicity
	cone associated with the determinant restricted to the $2d+1$-dimensional
	subspace of Hankel matrices. Its extreme rays have the form
	\[ \phi_{2,d}(x,y)\phi_{2,d}(x,y)' = \begin{bmatrix} x^d\\x^{d-1}y\\\vdots\\ xy^{d-1}\\y^d\end{bmatrix}\begin{matrix}
	\begin{bmatrix} x^d & x^{d-1}y & \cdots & xy^{d-1}&y^d\end{bmatrix}\\
	\phantom{\begin{bmatrix} x^d & x^{d-1}y & \cdots & xy^{d-1}&y^d\end{bmatrix}}\\
	\phantom{\begin{bmatrix} x^d & x^{d-1}y & \cdots & xy^{d-1}&y^d\end{bmatrix}}\\
	\phantom{\begin{bmatrix} x^d & x^{d-1}y & \cdots & xy^{d-1}&y^d\end{bmatrix}}\\
	\phantom{\begin{bmatrix} x^d & x^{d-1}y & \cdots & xy^{d-1}&y^d\end{bmatrix}}
\end{matrix} \]
	and are rank one as symmetric matrices, {and therefore} have rank one with respect to the determinant polynomial. The
	cone ${\mathcal{H}}_{d+1}$
	is also linearly isomorphic to the cone ${\mathcal{C}_{2,2d}}$ over the homogeneous moment curve of degree $2d$, which is Terracini convex {from Corollary~\ref{cor:momentcurveterracini}}.
	As such, if we choose ${E^{(1)}},\ldots,{E^{(\ell)}}$ to be positive definite $(d+1)\times (d+1)$ Hankel matrices
	and if $1\leq \ell \leq d-2$, then the mixed discriminant of $({E^{(1)}},\ldots,{E^{(\ell)}},X,\ldots,X)$ restricted to
	Hankel matrices $X$ {yields an associated} hyperbolicity cone that is $(d-1-\ell)$-Terracini convex.
\end{example}

\section{Discussion}
\label{sec:discussion}

In this paper we introduced the notion of Terracini convex cones, generalizing
the notion of neighborly polyhedral cones to the non-polyhedral setting in a way that
includes examples such as the positive semidefinite cone and the cone over the moment curve.
This suggests the pursuit of a broader program that seeks to extend key notions
from polyhedral combinatorics to more general convex cones.

\paragraph{Explicit constructions}
	A significant feature of the literature on neighborly polytopes -- arguably, a principle reason for considering such polytopes in the first place -- is that they offer examples of various extremal polyhedral constructions.  Obtaining similar constructions with non-polyhedral Terracini convex cones would offer an interesting point of comparison with the polyhedral case.  For example, we are not aware whether the non-degeneracy and regularity conditions of Section~\ref{sec:neighborly} are necessary to conclude that $k$-neighborly cones are $k$-Terracini convex, and identifying potential counterexamples would provide an interesting extremal class of convex cones.  In a different direction, explicit constructions for linear images of the positive-semidefinite cone that are Terracini-convex would immediately yield explicit (non-random) families of linear maps for which the associated low-rank inverse problems considered in Section~\ref{sec:inv-psd} may be solved exactly via semidefinite programming; despite significant attention devoted to this question, we are not aware of any such families of linear maps.

\paragraph{Beyond generalizing neighborliness}
A simplicial polytope is one in which every proper face
is a simplex. In the spirit of this paper, a natural analogue in the non-polyhedral conic setting
would be a closed pointed convex cone for which every proper face is Terracini convex. Let us call
such convex cones \emph{boundary Terracini convex}.
Clearly the cone over any simplicial polytope is boundary Terracini convex, but
boundary Terracini convex cones are a much richer class.
One interesting example is the epigraph of the nuclear norm, i.e.,
$\{(X,t)\in \RR^{m\times m} \times \RR\;:\; \|X\|_{\star} \leq t\}$.
One can check that all of the proper faces of this convex cone are linearly isomorphic to
positive semidefinite cones. Moreover, we can deduce from~\cite[Corollary 17]{renegar2006hyperbolic}
that if $\Lambda_+$ is
a hyperbolicity cone that is boundary Terracini convex, then so are
derivative relaxations of $\Lambda_+$ (as long as they are pointed).
It would be interesting to study such boundary Terracini convex cones in more detail.

\paragraph{Weaker notions of Terracini convexity}
The key condition~\eqref{eq:kterracini} in the definition of $k$-Terracini convexity is required to hold for
every subset of at most $k$ extreme rays. It is natural to consider weaker notions of $k$-Terracini convexity
that only require~\eqref{eq:kterracini} to hold for `many' subsets of at most $k$ extreme rays.
By ruling out certain explicit configurations of $k$ extreme rays
such a definition would generalize important existing variations on neighborliness,
such as $k$-neighborly centrally symmetric polytopes (in which subsets of $k$ extreme points containing
an antipodal pair are excluded).
Another approach would be to require that~\eqref{eq:kterracini} hold for suitably generic subsets of at most
$k$ extreme rays. Seeking and studying examples of convex cones that are generically $k$-Terracini convex but
not $k$-Terracini convex, would lead to a deeper understanding of Terracini convexity and its variants.

\paragraph{Possible further constructions of $k$-Terracini convex cones}
We have seen that the positive semidefinite cone and the cone over the moment curve (or equivalently the
cone of Hankel positive semidefinite matrices) are Terracini convex. These are both
examples of spectrahedral cones (intersections of a positive semidefinite cone
with a subspace) with all extreme rays having rank one. This very special class of
spectrahedral cones were classified by Blehkerman, Sinn, and Velasco~\cite{blekherman2017sums}
and are closely connected to questions about the relationship between nonnegative
polynomials and sums of squares. It would be interesting to investigate the
Terracini convexity properties of spectrahedral cones with only rank one extreme rays.
Going one step further, one could similarly investigate the Terracini convexity properties
of hyperbolicity cones with only (hyperbolic) rank one extreme rays. Unlike the spectrahedral
setting, we are not aware of any nontrivial characterization of this class of convex cones.

Theorem~\ref{thm:most-tc} shows that with high probability, Gaussian
random linear images of the positive semidefinite cone are $k$-Terracini
convex, for a suitable $k$. The specific properties of the
positive semidefinite cone are only used in isolated places in the argument, and do not seem
to be essential. It is plausible that there an analogue of Theorem~\ref{thm:most-tc}
where the positive semidefinite cone is replaced with any Terracini convex hyperbolicity cone,
or perhaps even any Terracini convex cone. If this were the case, it would be a substantial further
generalization
of the fact that Gaussian random linear images of the simplex are $k$-neighborly polytopes, for suitable $k$
\cite{donohotanner2005neighborliness}. It would also suggest the broader applicability of the notion of Terracini
convexity for understanding convex relaxations of inverse problems.

\paragraph{Obstructions to lifts of convex sets} Another setting in which neighborliness is
useful, and Terracini convexity may find applications, is in the study of lifted representations of
convex sets. Given a convex set $\mathcal{C}$ and a closed convex cone $\mathcal{K}$, we say that
$\mathcal{C}$ has a $\mathcal{K}$-lift if we can express $\mathcal{C}$ as the linear
projection of an affine slice of $\mathcal{K}$. Such representations of convex sets are of importance
when convex optimization problems are expressed in conic form. In particular, they play a prominent
role in the study of the expressive power of linear, second-order cone, and semidefinite programming of
a given size (see, e.g.,~\cite{fawzi2020lifting} for a recent survey). It turns out if a convex set satisfies
a notion of $k$-neighborliness that is somewhat weaker than that studied in this paper, then it cannot have
a $\mathcal{K}$-lift where $\mathcal{K}$ is a product of finitely many $k\times k$ positive
semidefinite cones~\cite{averkov2019optimal}.
A natural extension of this line of inquiry would be to investigate whether Terracini convexity
properties also provide obstructions to the existence of certain lifted representations of convex sets.

\section*{Acknowledgements}
The authors would like to thank Rainer Sinn for helpful conversations.  V.C. was supported in part by National Science Foundation grant CCF-1637598, in part by National Science Foundation grant DMS-2113724, 
and in part by AFOSR grant FA9550-20-1-0320. J.\ S.\ was supported in part by an Australian Research Council 
Discovery Early Career Researcher Award (project number DE210101056) funded by the Australian Government.

\bibliographystyle{alpha}
\bibliography{terracini-bib}

\end{document}